\documentclass[12pt, reqno]{amsart}
\usepackage{color}
\usepackage{hyperref}
\usepackage{graphicx}
\usepackage{amsmath}
\usepackage{amsthm}
\usepackage{amssymb,bbm}%
\usepackage[numbers, square]{natbib}
\usepackage{mathrsfs}

  \evensidemargin
\oddsidemargin
\numberwithin{equation}{section}

\newcommand{\sph}{\text{sph}}

\newcommand{\cns}{c_{\text{NS}}} 

\newcommand{\E}{\mathbb E}
\newcommand{\proj}{\zeta}
\newcommand{\bS}{\mathbb S}
\newcommand{\bA}{\mathbb A}

\newcommand{\R}{\mathbb{R}}
\newcommand{\N}{\mathbb{N}}

\newcommand{\C}{\mathbb{C}}
\newcommand{\Z}{\mathbb{Z}}

\renewcommand{\P}{\mathbb{P}}
\newcommand{\Q}{\mathbb{Q}}

\renewcommand{\Im}{\operatorname{Im}}  
\renewcommand{\Re}{\operatorname{Re}}

\newcommand{\cP}{\mathcal{P}}
\newcommand{\cF}{\mathcal{F}}
\newcommand{\cG}{\mathcal{G}}

\newcommand{\Var}{\mathop{\mathrm{Var}}\nolimits}

\newcommand{\eps}{\varepsilon}
\newcommand{\eqdistr}{\stackrel{d}{=}}
\newcommand{\todistr}{\overset{d}{\underset{n\to\infty}\longrightarrow}}
\newcommand{\toweak}{\overset{w}{\underset{n\to\infty}\longrightarrow}}

\newcommand{\toas}{\overset{a.s.}{\underset{n\to\infty}\longrightarrow}}
\newcommand{\toasR}{\overset{a.s.}{\underset{R\to\infty}\longrightarrow}}
\newcommand{\toR}{\overset{}{\underset{R\to\infty}\longrightarrow}}
\newcommand{\ton}{\overset{}{\underset{n\to\infty}\longrightarrow}}
\newcommand{\tok}{\overset{}{\underset{k\to\infty}\longrightarrow}}

\newcommand{\ind}{\mathbbm{1}}

\newcommand{\dd}{{\rm d}}
\newcommand{\eee}{{\rm e}}

\renewcommand{\Re}{\operatorname{Re}}  
\renewcommand{\Im}{\operatorname{Im}}  

\theoremstyle{plain}
\newtheorem{theorem}{Theorem}[section]
\newtheorem{proposition}[theorem]{Proposition}
\newtheorem{lemma}[theorem]{Lemma}
\newtheorem{corollary}[theorem]{Corollary}
\theoremstyle{definition}

\newtheorem{remark}[theorem]{Remark}

\newcommand{\nod}{\mathcal{N}}
\newcommand{\area}{\operatorname{Area}}

\begin{document}

\author{Zakhar Kabluchko}
\address{Zakhar Kabluchko: Institut f\"ur Mathematische Stochastik,
Westf\"alische Wilhelms-Universit\"at M\"unster,
Orl\'eans-Ring 10,
48149 M\"unster, Germany}
\email{zakhar.kabluchko@uni-muenster.de}

\author{Igor Wigman}
\address{Igor Wigman: Department of Mathematics, King’s College London, UK}
\email{igor.wigman@kcl.ac.uk}

\title[]{Asymptotics for the expected number of nodal components for random lemniscates}
\keywords{}
\subjclass[2010]{Primary, 30C15, 14P05; secondary, 14P25, 60G60, 60G15}
\thanks{}
\begin{abstract}
We determine the true asymptotic behaviour for the expected number of connected components for
a model of random lemniscates proposed recently by Lerario and Lundberg. These are defined as the subsets of the Riemann sphere, where the absolute value of certain random, $\text{SO}(3)$-invariant rational function of degree $n$ equals to $1$. We show that the expected number of the connected components of these lemniscates, divided by $n$, converges to a positive constant defined in terms of the quotient of two independent plane Gaussian analytic functions. A major obstacle in applying the novel non-local techniques due to Nazarov and Sodin on this problem is the underlying non-Gaussianity, intristic to the studied model.

\end{abstract}

\maketitle

\section{Introduction}\label{sec:introduction}

\subsection{Random lemniscates}

Let $\xi_0,\xi_1,\ldots$ and $\eta_0,\eta_1,\ldots$ be independent random variables having the standard complex Gaussian distribution. That is, their real and imaginary parts are independent centred real Gaussian random variables with variance $1/2$. Consider the so-called {\em spherical random polynomials} defined by
\begin{equation}\label{eq:def_p_n}
p_n(z) = \sum_{k=0}^n \xi_k \sqrt{\binom nk} z^k,
\quad
q_n(z) = \sum_{k=0}^n \eta_k \sqrt{\binom nk} z^k,
\quad
z\in\C.
\end{equation}
In the following, we shall frequently identify the extended complex plane $\C\cup\{\infty\}$ with the centred unit sphere $\bS^2$ in $\R^3$ by means of the stereographic projection $\proj: \bS^2 \to \C\cup\{\infty\}$ given by
\begin{equation}
\label{eq:zeta ster proj def}
\proj (u,v,w) = \left(\frac{u}{1-w},\frac{v}{1-w}\right),
\quad
u^2+v^2+w^2 =1, \; (u,v,w)\neq (0,0,1),
\end{equation}
and $\proj(0,0,1) = \infty$.
Lerario and Lundberg~\cite{lerario_lundberg} introduced the random rational function $\Psi_n:\bS^2 \to \C\cup\{\infty\}$ given by
\begin{equation}
\label{eq:Psin def}
\Psi_n(x) := \frac{p_n(\proj(x))}{q_n(\proj(x))}, \quad x\in\bS^2,
\end{equation}
and studied what they called ``{\em random lemniscate}''
\begin{equation}
\label{eq:Gamman lemn def}
\Gamma_n:= \{x\in\bS^2\colon |\Psi_n(x)| = 1\}.
\end{equation}
Both the probability law of the random rational function $\Psi_n$ and the nodal set $\Gamma_n$ are invariant ~\cite{lerario_lundberg} with respect to the natural action of $\text{SO}(3)$ on $\bS^2$. Using the Kac-Rice formula, Lerario and Lundberg~\cite{lerario_lundberg}  showed that the expected spherical length of $\Gamma_n$ equals $(\pi^2/2)\sqrt n$. The length and the number of components of random lemniscates associated to another family of random polynomials, the Kac polynomials, were studied by~\citet{lundberg_ramachandran}.

\subsection{Background and statement of the main result}

The key object of this manuscript is the number of connected components of $\Gamma_n$, denoted by $\nod(\Gamma_n)$,
more precisely, the asymptotic law of its expectation $\E[\nod(\Gamma_n)]$. Nazarov and Sodin ~\cite{NaSoAJM} have introduced a powerful machinery,
allowing for the {\em precise} asymptotic analysis of the expected number of connected components of the zero set of a particular Gaussian ensemble of spherical random field, and  further developed it ~\cite{sodin_lec_notes,nazarov_sodin} and abstracted their methods to treat a general class of Gaussian random fields and Gaussian ensembles of functions possessing some natural notion of {\em scaling}. These ideas have been employed by Sarnak and Wigman ~\cite{sarnak_wigman16}
to measure more refined quantities, like the number of connected components belonging to a given topological class, and their mutual positions
(``nesting"). Further, various upper and lower bounds were established for the important Kostlan ensemble of random polynomials defined
on the real (and complex) projective space (``statistical version of Hilbert's $16$th problem"), and their generalisations; see, e.g.~\cite{gayet2012betti,gayet2014lower,lerario2014statistics}.

The Gaussian assumption is {\em essential} to all of the above works, and, to our best knowledge, none of these is applicable to a non-Gaussian
situation, like ours \eqref{eq:Psin def}, i.e.\ estimating the expected number of connected components $\nod(\Gamma_n)$ of
the random lemniscates \eqref{eq:Gamman lemn def}.
For this model, Lerario and Lundberg~\cite{lerario_lundberg} proved the bounds $c_1 n \leq \E \left[ \nod(\Gamma_n)\right]\leq n$
for some constant $0<c_1 < 1$. Our main result determines the true asymptotic behaviour of the expected number of
connected components of $\Gamma_{n}$.

\begin{theorem}\label{theo:main}
There exists a positive constant $0<\cns\leq 1/\pi$, so that we have
\begin{equation}
\label{eq:c0 lim const sphere}
\lim_{n\to\infty} \frac{\E\left[ \nod(\Gamma_n)\right]}{n} = \pi \cns.
\end{equation}
\end{theorem}

The number $\cns$ is the Nazarov-Sodin type constant related to
nodal components of certain random meromorphic function to be defined below.

\subsection{Outline of the paper}

Here we introduce the principal notation and state preparatory results towards the proof of Theorem~\ref{theo:main}.

\subsubsection{The stationary meromorphic function}

The {\em plane Gaussian entire function} (GEF) is defined as
\begin{equation}
\label{eq:G GAF def}
G(z):= \sum_{n=0}^\infty \xi_n \frac{z^n}{\sqrt{n!}},
\quad
z\in \C,
\end{equation}
where $\xi_0,\xi_1,\ldots$ are i.i.d.\ standard complex Gaussian random variables. This means that $\Re \xi_k$ and $\Im \xi_k$ are independent, centred real Gaussian with variance $1/2$. The complex, mean zero Gaussian random field $(G(z))_{z\in\C}$ is uniquely defined by its covariance function
\begin{equation}
\label{eq:rG covar GEF def}
r_{G}(z,w):=\E [G(z) \cdot \overline{G(w)}] = \eee^{z\bar w},
\quad
\E [G(z)\cdot G(w)] = 0,
\quad
z,w\in\C.
\end{equation}

We are interested in the random meromorphic function $\Psi: \C\to \C \cup \{\infty\}$, defined as follows.
Let $(G_1(z))_{z\in\C}$ and $(G_2(z))_{z\in\C}$ be two independent plane Gaussian analytic functions \eqref{eq:G GAF def}. Then, we define
\begin{equation}
\label{eq:Psi def}
\Psi(z) := \frac {G_1(z)}{G_2(z)}, \quad z\in \C.
\end{equation}
In Section~\ref{sec:Psi_properties} we shall  prove that $\Psi$ is stationary, isotropic and mixing (see Lemma~\ref{lem:GAF_transform} and Proposition~\ref{prop:mixing}); we shall also establish that it is the scaling limit of $\Psi_{n}$ in
\eqref{eq:Psin def},
after appropriate scaling (see Proposition \ref{prop:functional_CLT}). The stationarity of $|G(z)|^2/\E\left[ |G(z)|^2\right]$ and some related properties will be established in Lemma~\ref{lem:cond dist a+bi}.

\subsubsection{Notation for discs and caps}
The open and the closed discs in $\C$ will be denoted by
$$
D(x_0;R):=\{z\in \C: |z- x_{0}|<R\}
\quad \text{ and } \quad
\bar D(x_0;R):=\{z\in \C: |z-x_0|\leq R\},
$$
where $x_0\in\C$ is the center of the disc and $R>0$ is its radius.
The geodesic distance on the unit sphere $\bS^2$ will be denoted by $\rho$, and the open and closed spherical caps of radius $r>0$ centred at $x_0\in \bS^2$ will be denoted by
\begin{align*}
B(x_0;r) := \{y\in\bS^2 \colon \rho(x_0,y)<r\},
\quad \text{ and } \quad
\bar B(x_0;r) := \{y\in\bS^2 \colon \rho(x_0,y)\leq r\},
\end{align*}
respectively.

\subsubsection{Nazarov-Sodin type constant \texorpdfstring{$\cns$}{cNS}}
Nazarov and Sodin~\cite{nazarov_sodin} (see also~\cite{sodin_lec_notes}) studied the connected components of random sets of the form $\{x\in\R^d\colon \xi(x) = 0\}$, where $(\xi(x))_{x\in\R^d}$ is a stationary real Gaussian process.
Among other results, they proved in~\cite[Theorem~1]{sodin_lec_notes} a law of large numbers for the number of connected components of these nodal sets.
We are going to state a similar result for the connected components of the  random set
$$
\Gamma = \{z\in\C\colon |\Psi(z)|=1\}.
$$
We write $\nod(\Gamma;R)$ for the number of connected components of  $\Gamma$ completely contained in the open centred disc $D(0;R)=\{z\in \C: |z|<R\}$.
\begin{theorem}
\label{theo:cns_existence}
There exists a constant $\cns>0$ such that
\begin{equation}
\label{eq:cns existence}
\frac {\nod(\Gamma;R)}{\pi R^2}  \toasR \cns
\quad \text{ and } \quad
\lim_{R\to\infty} \frac {\E\left[\nod(\Gamma;R)\right]}{\pi R^2} = \cns.
\end{equation}
\end{theorem}

As a by-product of the proof of Theorem \ref{theo:main}, it will also follow that the constant $\cns$ in \eqref{eq:cns existence} is the same as in \eqref{eq:c0 lim const sphere}.

\subsubsection{Estimates for small components}
Nodal domains are connected components of the complement $\C\backslash \Gamma$ of the lemniscate $\Gamma = \{z\in\C: |\Psi(z)|=1\}$.
For some $\delta>0$, a connected component of $\Gamma$ is called $\delta$-{\em small}, if it is adjacent to a nodal domain of area $<\delta$. Otherwise, it is called $\delta$-large.
Let $\nod_{\delta}(\Gamma;R)$ be the total number of $\delta$-large connected components of $\Gamma$ lying entirely in the radius $R$ centred disc $D(0;R)$. Also, let $$\nod_{\delta-sm}(\Gamma;R) :=\nod(\Gamma;R)-  \nod_{\delta}(\Gamma;R)$$ be the number of $\delta$-small components lying entirely in the same disc $D(0;R)$.  In the following proposition we give an upper bound on the expected number of small components.

\begin{proposition}
\label{prop:sm do small}
There exist constants  $C>0$ and $c_0>0$ such that for all $R>1$, $\delta>0$,
\begin{equation}
\label{eq:exp(delta-sm)<=Cdelta^c0R^2 inf}
\E\left[\nod_{\delta-sm}(\Gamma;R)\right] \le C\delta^{c_{0}}\cdot R^{2}.
\end{equation}
\end{proposition}

We shall also need a similar estimate for the small components of finite-degree lemniscates $\Gamma_n = \{x\in\bS^2: |\Psi_n(x)|=1\}$. For $\delta>0$, a connected component of $\Gamma_n$ is called $\delta/n$-{\em small}, if it is adjacent to a domain of (spherical) area $<\delta/n$. Otherwise, it is called $\delta/n$-large. Let $\nod_{\delta/n}(\Gamma_n)$ (respectively, $\nod_{\delta/n-sm}(\Gamma_n)$), be the total number of $\delta/n$-large (respectively, $\delta/n$-small) connected components of $\Gamma_n$. Also, for $r>0$,
analogously to the above, let $\nod(\Gamma_{n}; x_0, r)$, $\nod_{\delta/n}(\Gamma_{n}; x_0, r)$ and $\nod_{\delta/n-sm}(\Gamma_{n}; x_0, r)$
be the total number, the number of $\delta/n$-large, and the number of $\delta/n$-small, components of $\Gamma_{n}$
lying entirely in the spherical cap $B(x_{0};r)$ respectively.

\begin{proposition}
\label{prop:sm do small_finite_deg}
There
exist constants $C_1>0$ and $c_0>0$ such that
$$
\E\left[\nod_{\delta/n-sm}(\Gamma_{n}; x_0, R/\sqrt{n})\right] \le C_{1}\delta^{c_{0}} \cdot R^2
$$
uniformly for all $n\in\N$, $x_0\in \bS^2$, $1<R<\sqrt{n}$ and $\delta>0$.
\end{proposition}

The proofs of propositions~\ref{prop:sm do small} and~\ref{prop:sm do small_finite_deg} are postponed till Section~\ref{sec:small_components}.
The proof of Proposition \ref{prop:sm do small_finite_deg} also yields the {\em global} version
$$\E\left[\nod_{\delta/n-sm}(\Gamma_{n})\right] \le C_{1}\delta^{c_{0}} \cdot n$$
of \eqref{prop:sm do small_finite_deg} (cf. \eqref{eq:small dom ineq det1 glob}).

\subsection*{Acknowledgements}

We are grateful to P.\ Sarnak and M.\ Sodin for some fruitful conversations on the presented research, and,
in particular, concerning the small components estimates for non-Gaussian random fields.
The research leading to these results has received funding from
the German Research Foundation under Germany's Excellence Strategy  EXC 2044 -- 390685587, Mathematics M\"unster: Dynamics - Geometry - Structure (Z.K.) and
the European Research Council under the European Union’s Seventh Framework Programme
(FP7/2007-2013) / ERC grant agreements n$^{\text{o}}$ 335141 Nodal (I.W.).

\section{Properties of the random meromorphic function}
\label{sec:Psi_properties}

\subsection{Stationarity properties of \texorpdfstring{$\Psi$}{Psi}}
We recall that $\Psi(z)= G_1(z)/G_2(z)$, where $G_1$ and $G_2$ are two independent GEF's as in~\eqref{eq:G GAF def}, and $\proj$ is the stereographic projection \eqref{eq:zeta ster proj def}.
The basic properties of $\Psi$ are collected in the following lemma. Its first part states that although the Gaussian entire function $G$ is not stationary, it is well-behaved under time shifts.
\begin{lemma}\label{lem:GAF_transform}

\begin{enumerate}

\item Let $a,b\in\C$ with $|a|=1$. Then, we have the following distributional equality of random fields:
$$
\left(\frac{G(az+b)}{\eee^{\frac 12|b|^2 + a \bar b z}}\right)_{z\in\C}
\eqdistr
(G(z))_{z\in\C}.
$$
\item The random field $\Psi$ is stationary and isotropic, that is for every $a,b\in\C$ with $|a|=1$ we have
$$
(\Psi(az+b))_{z\in\C} \eqdistr (\Psi(z))_{z\in\C}.
$$
\item For every rotation $g\in SO(3)$ of the unit sphere $\bS^2$, we have
$$
(g \proj^{-1} \Psi(z))_{z\in\C} \eqdistr ( \proj^{-1}\Psi(z))_{z\in\C}.
$$
\end{enumerate}
\end{lemma}
\begin{proof}

For the first statement of Lemma \ref{lem:GAF_transform}, we notice that, since both random fields are mean zero complex Gaussian, it suffices to check the equality of covariance functions. For arbitrary $z,w\in\C$ we have
\begin{equation}\label{eq:cov_G_normalized}
\E \left[ \frac{G(az+b)}{\eee^{\frac 12|b|^2 + a \bar b z}} \cdot \overline{\left(\frac{G(aw+b)}{\eee^{\frac 12|b|^2 + a \bar b w}}\right)}\right]
=
\E \left[ \frac{G(az+b)\cdot \overline{G(aw+b)}}{\eee^{|b|^2 + a \bar b z + \bar a  b \bar w}} \right]
=
\frac{\eee^{(az+b)\overline{(aw+b)}}}{\eee^{|b|^2 + a \bar b z + \bar a  b \bar w}}
= \eee^{z\bar w},
\end{equation}
where we used that $a\bar a= |a|^2 = 1$. This proves the first claim of Lemma \ref{lem:GAF_transform}.

To prove the second claim of Lemma \ref{lem:GAF_transform}, write
$$
\Psi(az+b) = \frac{G_1(az+b)}{G_2(az+b)} = \frac{G_1(az+b)/\eee^{\frac 12|b|^2 + a \bar b z}}{G_2(az+b)/\eee^{\frac 12|b|^2 + a \bar b z}},
$$
and the statement follows from the first claim of Lemma \ref{lem:GAF_transform}. To prove the third one,
we need to show that
$$
\left(g\proj^{-1}\frac{G_1(z)}{G_2(z)}\right)_{z\in\C}
\eqdistr
\left(\proj^{-1}\frac{G_1(z)}{G_2(z)}\right)_{z\in\C}
$$
or, equivalently,
$$
\left(\proj g\proj^{-1}\frac{G_1(z)}{G_2(z)}\right)_{z\in\C}
\eqdistr
\left(\frac{G_1(z)}{G_2(z)}\right)_{z\in\C}.
$$
Note that $\proj g \proj^{-1}: \C \cup\{\infty\} \to \C \cup\{\infty\}$ is a fractional-linear transformation of the form
$$
\proj g \proj^{-1} (z) = \frac{\lambda z + \mu}{-\bar \mu z + \bar \lambda}, \quad \lambda,\mu\in\C, \quad |\lambda|^2 + |\mu|^2=1.
$$
Clearly,
$$
\proj g\proj^{-1}\frac{G_1(z)}{G_2(z)}= \frac{\lambda \frac{G_1(z)}{G_2(z)} + \mu}{-\bar \mu \frac{G_1(z)}{G_2(z)} + \bar \lambda}
=
\frac{\lambda G_1(z) + \mu G_2(z)}{- \bar \mu G_1(z) + \bar \lambda G_2(z)} =: \frac{H_1(z)}{H_2(z)},
$$
where $H_1(z) = \lambda G_1(z) + \mu G_2(z)$ and $H_2(z)= - \bar \mu G_1(z) + \bar \lambda G_2(z)$. Now, $H_1$ and $H_2$, as well as their joint multivariate distributions, are complex Gaussian with zero mean and
$$
\E [H_1(z) \overline{H_1(w)}] = \E [H_2(z) \overline{H_2(w)}] = \eee^{z\bar w}
$$
and
$$
\E [H_1(z) \overline{H_2(w)}] = \E[(\lambda G_1(z) + \mu G_2(z)) (- \mu \overline{G_1(z)} +  \lambda \overline{G_2(z)})] = 0.
$$
Hence, the joint distribution of the pair $(H_1,H_2)$ is the same as of the pair $(G_1,G_2)$, which proves the claim.
\end{proof}

\begin{remark}
Since the uniform distribution is the only rotationally invariant distribution on $\bS^2$, we conclude that for every $z\in\C$, $\proj^{-1}\Psi(z)$ is uniform on $\bS^2$. Equivalently, $$\Psi(z) = \frac{G_1(z)}{G_2(z)}$$ has probability density $\pi^{-1}(1+|z|^2)^{-2}$ w.r.t.\ the Lebesgue measure on $\C$, which could also be established by a direct computation.
\end{remark}
\begin{remark}\label{rem:random_meromorphic_CPd}
Generalizing the above approach, it is possible to define a natural random holomorphic map from $\C$ to $\mathbb{CP}^d$. Recall that the complex projective space $\mathbb{CP}^d$ consists of all tuples $[z_1,\ldots,z_{d+1}]\in \mathbb C^d$,   $[z_1,\ldots,z_{d+1}]\neq (0,\ldots,0)$,  where the tuples $[z_1,\ldots,z_{d+1}]$ and $[\lambda z_1,\ldots, \lambda z_{d+1}]$  are considered to be equivalent for $\lambda \in \C\backslash\{0\}$. Let now $$(G_1(z))_{z\in\C}, \ldots, (G_{d+1}(z))_{z\in\C}$$ be independent copies of the Gaussian analytic function $(G(z))_{z\in\C}$. Consider a random holomorphic map $\Psi:\C \to \mathbb{CP}^d$ defined by
$$
\Psi(z) = [G_1(z),\ldots, G_{d+1}(z)], \quad z\in\C.
$$
Following along the above arguments, it is possible to show that
$$
(\Psi(az+b))_{z\in\C} \eqdistr (\Psi(z))_{z\in\C}
\quad\text{ and }\quad
(U (\Psi(z)))_{z\in\C} \eqdistr (\Psi(z))_{z\in\C}
$$
for every $a,b\in\C$ with $|a|=1$ and for every unitary transformation $U:\C^{d+1} \to \C^{d+1}$. Moreover, for every $z\in\C$, $\Psi(z)$ is distributed on $\mathbb{CP}^d$ according to the Fubini-Study volume (normalized to be a probability measure).
\end{remark}

\subsection{Mixing}
\label{sec:mixing}
Let $\bA$ be the space of all holomorphic maps from $\C$ to $\bS^2$ (any such map can be naturally identified with a meromorphic function) endowed with the topology of uniform convergence on compact sets. The unit sphere $\bS^2$ is endowed with the usual geodesic metric. We denote
by $\mathscr B(\bA)$ the Borel $\sigma$-algebra generated by the open subsets of $\bA$. For each $t\in\C$ consider the map $T_t:\bA\to\bA$ defined by shifting the function by $t$:
$$
T_t f(z) = f(z-t), \quad f\in \bA.
$$
It is clear that $T_0$ is the identity map and that $T_t T_s = T_{t+s}$ for all $t,s\in \C$, making $(T_{t})_{t\in\C}$ a flow.
Also, $(t, f)\mapsto T_t f$ defines a continuous map from $\C\times \bA$ to $\bA$, which implies that the flow is measurable.
Let furthermore $\Q$ denote the probability law of $\proj^{-1}\Psi$, where $\Psi$ is the random meromorphic function defined in~\eqref{eq:Psi def} above, that is $\Q$ is a probability measure on $(\bA, \mathscr B(\bA))$ defined by $\Q[B] = \P[\proj^{-1}\Psi\in B]$ for every Borel set $B\in \mathscr B(\bA)$. By Lemma
\ref{lem:GAF_transform}, $(T_t)_{t\in\C}$ is a measure-preserving flow on the probability space $(\bA, \mathscr B(\bA), \Q)$.
\begin{proposition}\label{prop:mixing}
The flow $(T_t)_{t\in\C}$ is mixing, that is for every events $A,B \in \mathscr B(\bA)$ we have
$$
\lim_{|t|\to\infty} \Q[(T_t A) \cap B] = \Q[A] \cdot \Q[B].
$$
\end{proposition}

Proposition \ref{prop:mixing} implies, in particular, the ergodicity of the flow $(T_t)_{t\in\C}$, sufficient for our needs, for the purpose of
giving a proof for Theorem \ref{theo:cns_existence} (see Section~\ref{sec:NS lim scal thm proof} below). To prove Proposition \ref{prop:mixing}
we require the following well-known lemma; c.f.~\cite[Theorem 1.17 on p.~41]{walters_book}). For the sake of the reader's convenience we include its standard proof immediately below.

\begin{lemma}\label{lem:mixing}
Let $(T_t)_{t\in\R^d}$ be a measure-preserving flow on a probability space $(E, \cF, \mu)$ such that, for all $A,B$ in a system of sets $\cP\subset \cF$, we have
$$
\lim_{|t|\to\infty} \mu \{T_t A \cap B\} =  \mu\{A\}\cdot \mu\{B\}.
$$
If the system $\cP$ is a $\pi$-system (i.e.,\ for $A,B\in \cP$ we also have $A\cap B\in \cP$) and the $\sigma$-algebra generated by $\cP$ is $\cF$, then $(T_t)_{t\in\R^d}$ is mixing.
\end{lemma}

\begin{proof}[Proof of Lemma~\ref{lem:mixing}]
Consider the class $\cG$ of all sets $A\in \cF$ such that $$\lim_{|t|\to\infty} \mu \{T_t A \cap B\} =  \mu\{A\}\cdot \mu\{B\}$$ for all $B\in \cP$. This class contains $\cP$ and is a $\lambda$-system, that is it is closed under taking complements and countable \textit{disjoint} unions. To prove the latter claim, let $A_1,A_2\ldots\in\cG$ be a pairwise disjoint countable collection of sets.
Then $$\lim_{|t|\to\infty} \mu \{T_t A_k \cap B\} = \mu\{A_k\} \cdot \mu\{B\}$$ for every $k\in\N$, $B\in\cP$. Since $\mu \{T_t A_k \cap B\}\leq \mu\{A_k\}$, and $\sum_{k=1}^\infty \mu\{A_k\} \leq 1$, we can apply the Dominated
Convergence Theorem (with sums regarded as special case of integrals) to conclude that
$$
\mu \left\{T_t \left(\cup_{k=1}^\infty A_k\right) \cap B\right\} =
\sum_{k=1}^\infty \mu \{T_t A_k \cap B\}
\to \sum_{k=1}^\infty \mu\{A_k\}\cdot \mu\{B\}
=\mu\left\{\cup_{k=1}^\infty A_k\right\}\cdot \mu\{B\},
$$
as $|t|\to\infty$,
thus proving that $\cG$ is a $\lambda$-system.

By the $\pi$-$\lambda$-theorem, $\cG$ includes the $\sigma$-algebra generated by $\cP$, which is $\cF$. The given argument shows that
\begin{equation}
\label{eq:mix A,B restr B}
\lim_{|t|\to\infty} \mu \{T_t A \cap B\} =  \mu\{A\} \mu\{B\}
\end{equation}
for all $A\in \cF$, $B\in\cP$. Repeating the same argument with the roles of $A$ and $B$ reversed, one easily shows that, in fact, \eqref{eq:mix A,B restr B} holds for all $B\in\cF$.
\end{proof}

\begin{proof}[Proof of Proposition~\ref{prop:mixing}]
We prove the claim for sets $A$ and $B$ of the following form:
\begin{align*}
A &= \{f\in \bA\colon (f(t_1), \ldots, f(t_d)) \in E\},\\
B &= \{f\in \bA\colon ( f(s_1), \ldots, f(s_d)) \in F\},
\end{align*}
where $E\subset (\bS^2)^d$ and $F\subset (\bS^2)^d$ are open subsets whose boundary has zero measure.
By Lemma~\ref{lem:mixing}, this implies mixing.
We have
\begin{align*}
\Q [T_t A\cap B]
&=
\P\left[\left(\proj^{-1}\left(\frac{G_1(t_k+t)}{G_2(t_k+t)}\right)\right)_{k=1}^d \in E, \left(\proj^{-1}\left(\frac{G_1(s_k)}{G_2(s_k)}\right)\right)_{k=1}^d \in F \right]\\
&=
\P\left[\left(\proj^{-1}\left(\frac{G_1(t_k+t)/ \eee^{\frac 12 |t|^2 + \bar t t_k}}{G_2(t_k+t)/ \eee^{\frac 12 |t|^2 + \bar t t_k}}\right)\right)_{k=1}^d \in E, \left(\proj^{-1}\left(\frac{G_1(s_k)}{G_2(s_k)}\right)\right)_{k=1}^d \in F \right].
\end{align*}
We now consider the complex Gaussian random vector with the following $4d$ components:
\begin{equation}\label{eq:vector_G1_G2}
\frac{G_1(t_k+t)}{\eee^{\frac 12 |t|^2 + \bar t t_k}},
\;\;
\frac{G_2(t_k+t)}{\eee^{\frac 12 |t|^2 + \bar t t_k}},
\;\;
G_1(s_k),
\;\;
G_2(s_k),
\quad k=1,\ldots,d.
\end{equation}
Let us look at the covariance matrix of this vector as $|t|\to\infty$. First of all, for every $t\in\C$ we have
\begin{equation*}
\begin{split}
\E \left[\frac{G_1(t_k+t)}{\eee^{\frac 12 |t|^2 + \bar t t_k}}\cdot \overline{\frac{G_2(t_j+t)}{\eee^{\frac 12 |t|^2 + \bar t t_j}}} \right]
&=
\E[G_1(s_k)\cdot \overline{G_2(s_j)}]
\\&=
\E \left[\frac{G_1(t_k+t)}{\eee^{\frac 12 |t|^2 + \bar t t_k}}\cdot \overline{G_2(s_j)}\right]
=
\E \left[\frac{G_2(t_k+t)}{\eee^{\frac 12 |t|^2 + \bar t t_k}} \cdot\overline{G_1(s_j)}\right]
=  0
 \end{split}
\end{equation*}
by the independence of $G_1$ and $G_2$. Further, it easily follows from~\eqref{eq:cov_G_normalized} that
$$
\E \left[\frac{G_1(t_k+t)}{\eee^{\frac 12 |t|^2 + \bar t t_k}} \cdot \overline{\left(\frac{G_1(t_j+t)}{\eee^{\frac 12 |t|^2 + \bar t t_j}}\right)} \right] = \eee^{t_k \overline{t_j}}.
$$
Finally, as $|t|\to\infty$ we have
\begin{align*}
\E \left[\frac{G_1(t_k+t)}{\eee^{\frac 12 |t|^2 + \bar t t_k}} \cdot\overline{G_1(s_j)}\right]
=
\E \left[\frac{G_2(t_k+t)}{\eee^{\frac 12 |t|^2 + \bar t t_k}} \cdot\overline{G_2(s_j)}\right]
= \eee^{(t_k+t)\overline{s_j}- \frac 12 |t|^2 - \bar t t_k } \to 0.
\end{align*}
Since the vector~\eqref{eq:vector_G1_G2} is complex Gaussian, it follows that, as $|t|\to\infty$, it converges weakly to the complex Gaussian vector with $4d$ components
$$
H_1(t_k), \;\; H_2(t_k), \;\; G_1(s_k), \;\; G_2(s_k), \quad  k=1,\ldots,d,
$$
where $(H_1(z))_{z\in\C}$ and $(H_2(z))_{z\in\Z}$ are mutually independent copies of $G$ that are independent of everything else. Thus, by the Portmanteau theorem,
\begin{align*}
\lim_{|t|\to\infty} \Q [T_t A\cap B]
&=
\P\left[\left(\proj^{-1}\left(\frac{H_1(t_k)}{H_2(t_k)}\right)\right)_{k=1}^d \in E, \left(\proj^{-1}\left(\frac{G_1(s_k)}{G_2(s_k)}\right)\right)_{k=1}^d \in F\right]\\
&=
\P\left[\left(\proj^{-1}\left(\frac{H_1(t_k)}{H_2(t_k)}\right)\right)_{k=1}^d \in E\right] \cdot \P\left[\left(\proj^{-1}\left(\frac{G_1(s_k)}{G_2(s_k)}\right)\right)_{k=1}^d \in F\right]\\
&=
\Q[A] \cdot \Q[B],
\end{align*}
thus concluding the proof of Proposition \ref{prop:mixing}.
\end{proof}
\begin{corollary}
The flow $(T_t)_{t\in\C}$ is ergodic, that is the probability of every invariant set is $0$ or $1$.
\end{corollary}

\begin{remark}
Slightly generalizing the above proof one can show that $(T_t)_{t\in\C}$ is mixing of all orders, that is for every $m\in\N$ and every $A_1,\ldots,A_m\in\mathscr B(\bA)$,
$$
\lim_{\substack {|t_i -t_j|\to\infty\\\forall i\neq j}} \Q[T_{t_1} A_1  \cap \ldots \cap T_{t_m} A_m] = \Q[A_1]\cdot\ldots\cdot \Q[A_m].
$$
\end{remark}

\begin{remark}
Generalizing the above arguments, it is possible to show the ergodicity of the random holomorphic function $\Psi: \C\to \mathbb{CP}^d$ defined in Remark~\ref{rem:random_meromorphic_CPd} above. As a corollary, one can show that the surface $\Psi(\C)$ fills $\mathbb{CP}^d$ w.r.t. the Fubini-Study measure. Namely, let $\mu_n$ be the image measure (under $\Psi$) of the uniform  probability distribution on the square $[-n,n]^2$, that is $\mu_n$ is a (random) probability measure on $\mathbb {CP}^d$ defined by $\mu_n(A) = \frac 1 {4n^2}\text{Leb}\{\Psi^{-1}(A)\cap[-n,n]^2\}$. We claim that, as $n\rightarrow\infty$, $\mu_n$ converges weakly to the Fubini-Study probability measure on $\mathbb{CP}^d$, with probability $1$. To see this, take any continuous test function $f: \mathbb{CP}^d\to \R$ and observe that
$$
\int_{\mathbb{CP}^d} f\dd \mu_n = \frac 1 {4n^2} \int_{[-n,n]^2} f(\Psi(z)) |\dd z| \toas \E [f(\Psi(0))]
$$
by the ergodic theorem. Since $\Psi(0)$ is distributed according to the Fubini-Study probability measure, the claim follows.
\end{remark}

\subsection{Proof of Theorem \ref{theo:cns_existence}: establishing the Nazarov-Sodin type limit}
\label{sec:NS lim scal thm proof}

Recalling the notation of Section~\ref{sec:mixing},
we consider two commuting measure-preserving transformations of the probability space $(\bA, \mathscr B(\bA), \Q)$ given by backward shifting the meromorphic functions in horizontal and vertical directions, namely $\tau_1 = T_{-1}$ and $\tau_2 = T_{-i}$.
Let $\mathscr H$ be the family of all non-empty finite subsets of $\Z^2$, and
define the process $X=(X_A)_{A\in\mathscr H}$ indexed by $\mathscr{H}$ in the following way. Given a finite set $A\subset \Z^2$ and a meromorphic function $\omega\in \bA$ we let $X_A(\omega)$ to be the number of connected components of the set $\{z\in\C\colon |\omega(c)|=1\}$ lying entirely inside the finite union of squares $$A^*:=\{z\in\C\colon (\lfloor \Re z\rfloor, \lfloor \Im z\rfloor) \in A\}$$ with bottom left corner belonging to $A$.

Let us also recall the multiparameter subadditive ergodic theorem \cite[Theorem 2.14, page 210]{krengel_book}, we first explain in an abstract setting, with the broad plan of applying it on the said process $X$ within the specified context. Let $\tau_1$ and $\tau_2$ be two commuting measure-preserving automorphisms of a probability space. They generate a group $(\tau_u)_{u\in\Z^2}$ of measure preserving transformations given by $\tau_u = \tau_1^{u_1} \tau_2^{u_2}$ for $u=(u_1,u_2)\in\Z^2$. A family of integrable random variables $X=(X_{A})_{A\in \mathscr H}$ indexed by $\mathscr H$ is called a \emph{superadditive process} if it satisfies the following three conditions:
\begin{itemize}
\item[(a)] $X_A\circ \tau_u = X_{A+u}$ a.s.\ for all $A\in \mathscr H$ and $u\in\Z^2$.
\item[(b)] $X_{A\cup B} \geq X_{A}+X_{B}$ for all disjoint sets $A,B\in \mathscr H$.
\item[(c)] $\sup\left\{\E [X_A]/|A|\colon A\in \mathscr H\right\}<\infty$, where $|A|$ is the cardinality of $A$.
\end{itemize}
If $X$ as above is a superadditive process, then the superadditive ergodic theorem~\cite[Theorem 2.14, page 210]{krengel_book}\footnote{See also~\cite[p.~165]{nguyen} for the $L^1$-convergence, \cite{nguyen_zessin} for the notation used in~\cite{nguyen}, and~\cite{grillenberger_krengel}, where the strong subadditivity assumption is removed.} states that for every increasing sequence of convex bounded sets $V_1\subset V_2\subset \ldots \subset \R^2$
with
\begin{equation}
\label{eq:union Vk=R^2}
\bigcup\limits_{k=1}^{\infty} V_k =\R^2,
\end{equation}
we have
\begin{equation}
\label{eq:ergodicity imply concrete}
\frac{X_{V_{k}\cap \Z^2}}{|V_{k}\cap \Z^2|}  \tok X_\infty \quad \text{a.s.\ and in $L^1$}
\end{equation}
for some random variable $X_\infty$.

\begin{proof}[Proof of Theorem \ref{theo:cns_existence}]
Our aim is to show that
$$
\frac {\nod(\Gamma;R)}{\pi R^2}  \toR \cns \quad \text{a.s.\ and in $L^1$}
$$
for some constant $\cns>0$. To this end we plan to apply the aforementioned superadditive ergodic theorem,
yielding that \eqref{eq:ergodicity imply concrete} holds for an increasing sequence $\{V_{k}\}_{k\ge 1}$ of convex bounded sets
satisfying \eqref{eq:union Vk=R^2}, provided that the conditions (a)-(c) hold, which we verify next.
First, the superadditivity property (b) follows immediately from the definition, as does the property (a).

To verify condition (c), we could bound above the expected number of critical points and poles of $\Psi$ on a compact domain, also bounding the number of nodal components lying entirely inside that domain. Instead, we simplify the said procedure by arguing as follows.
We need to show that the expected number of connected components of $\Gamma$ in $A^*$ is bounded above by $C |A|$, for some absolute constant $C$. To every component $\gamma$ of $\Gamma$ we can assign the unique connected component (nodal domain) $G_\gamma$ of the set
$\{z\in\C\colon |\Psi(z)|\neq 1\}$, which is located inside $\gamma$ and is adjacent to $\gamma$, and it is clear that the map
$\gamma\mapsto G_\gamma$ is injective. We claim that necessarily $G_\gamma$ contains either a zero or a pole of $\Psi$. Indeed, the function $|\Psi(z)|$ equals $1$ on the boundary of $G_\gamma$. Assume that $G_\gamma$ contains no zeroes and no poles of $\Psi$. Then, both $\Psi$ and $1/\Psi$ are analytic on $G_\gamma$. By the maximum modulus principle applied to these functions, we have $|\Psi(z)|=1$ for all $z\in G_\gamma$, and then, by the Open Mapping Theorem, the function $\Psi$ must be constant. Since this event is of probability $0$, the expected number of connected components of $\Gamma$ inside $A^*$ is bounded above by the expected number of zeroes and poles of $\Psi$ in $A^*$, which is easily evaluated to be $(2/\pi) |A|$ by~\cite[Remark~2.4.5]{hough_book}.

In this setting, the superadditive ergodic theorem applied to the sequences of sets $V_k:= D(0;k-10)$ and $W_k:= D(0; k+10)$ states that
$$
\frac{X_{V_k\cap \Z^2}}{\pi k^2} \tok X_\infty
\quad \text{ and }\quad
\frac{X_{W_k\cap \Z^2}}{\pi k^2} \tok X_\infty
\quad \text{a.s.\ and in $L^1$},
$$
for some random variable $X_\infty$. The
limits of both sequences are equal because $V_k\subset W_k \subset V_{k+20}$.
Since $(V_{\lfloor R \rfloor}\cap \Z^2)^* \subset D(0;R) \subset (W_{\lfloor R \rfloor}\cap \Z^2)^*$, we arrive at
$$
\frac{\nod(\Gamma; R)}{\pi R^2} \toR X_\infty \quad \text{a.s.\ and in $L^1$}.
$$
The fact that the random variable $X_\infty$ is a.s.\ constant follows from the ergodicity of our process $\Psi$; see Proposition~\ref{prop:mixing}. Indeed, by a sandwich argument similar to the one used above, the random variable $X_\infty$ is invariant under the flow $(\tau_u)_{u\in\Z^2}$, that is $X_\infty \circ \tau_u = X_\infty$ a.s. for all $u\in\Z^2$.
Then, ergodicity implies that  $X_\infty$ a.s.\ equals to some constant that we denote by $\cns$.

The strict positivity of $\cns$ follows from the work of Lerario and Lundberg~\cite[Corollary on p.~653]{lerario_lundberg}
(upon bearing in mind the fact that the constants $c_{NS}$ in our theorems~\ref{theo:main} and~\ref{theo:cns_existence} are the same).
For the sake of the completeness of our arguments and the convenience of the reader we briefly sketch a different argument. For the positivity of $\cns$ it suffices to check that the expected number of connected components of $\Gamma$ entirely contained in $(0,1)^2$ is strictly positive. Take some deterministic analytic function $f:\C\to\C$ so that the set $\{|f|=1\}$ has at least one connected component contained in $(0,1)^2$. This property is shared by all analytic functions that are sufficiently close to $f$ in the uniform topology on $[0,1]^2$. It then remains to show that the probability that $\sup_{z\in [0,1]^2} |\Psi(z) - f(z)|<\eps$ is strictly positive for all $\eps>0$. This, in turn, would follow if we could show that the event $\sup_{z\in [0,1]^2} |G_1(z) - f(z)|<\eps$ is of positive probability (since the same argument could then be applied to show that $\sup_{z\in [0,1]^2} |G_2(z) - 1| < \eps$ with positive probability). To prove that $G_1$ can approximate $f$ arbitrarily well with positive probability, recall the Taylor series of $G_1$ given in~\eqref{eq:G GAF def}, note that any finite number $N$ of coefficients of this series can approximate the corresponding coefficients of $f$ arbitrarily well, and the probability that the tail of the series is smaller than $\eps/2$ on $[0,1]^2$ is positive. Altogether, this proves that $\cns>0$.

\end{proof}





\section{Proof of Theorem~\ref{theo:main}}

\subsection{Functional limit theorem}
Our aim is to prove a functional limit theorem which states that, for large $n$, the restriction of the function $\Psi_n(x)$ to the spherical cap $\bar B(x_0; R/\sqrt n)$ looks similarly to the restriction of the function $\Psi(z/2)$ to the disc $\bar D(0;R)$. Since these functions are defined on different domains, we have to construct a suitable map between these domains.   By the rotational invariance of $\Psi_n$, we may and will assume that $x_0=(0,0,-1)$ is the south pole of $\bS^2$. It is easy to check that the map
\begin{equation}\label{eq:map_cap_disk}
z\mapsto \zeta^{-1}\left(\frac{z\sin (R/\sqrt n)}{R(1+ \cos(R/\sqrt n))}\right)
\end{equation}
defines a bijection between $\bar D(0;R)$ and the spherical cap $\bar B(x_0; R/\sqrt n)$.
To define the functional space on which our weak convergence takes place, let $\bA(\bar D(0;R) \to \bS^2)$ be the space of functions with values in $\bS^2$ that are continuous on the closed disc $\bar D(0;R)$ and holomorphic in its interior. Endowed with the supremum metric, $\bA(\bar D(0;R) \to \bS^2)$ becomes a Polish space. Here, the unit sphere $\bS^2$ is endowed with the usual geodesic metric.

\begin{proposition}\label{prop:functional_CLT}
For every $R>0$, on the space $\bA(\bar D(0;R) \to \bS^2)$ we have the weak convergence
$$
\left((\zeta^{-1} \circ\Psi_n \circ\zeta^{-1})\left(\frac{z\sin (R/\sqrt n)}{R(1+ \cos(R/\sqrt n))}\right)\right)_{z\in \bar D(0;R)} \toweak ((\zeta^{-1}\circ\Psi)(z/2))_{z\in \bar D(0;R)},
$$
where $\Psi$ is the random meromorphic function defined in \eqref{eq:Psi def}.
\end{proposition}

To prove the proposition, we need first to show the weak convergence of the random polynomial $p_n$ to the Gaussian entire function $G$.
Let $\bA(\bar D(0;R) \to \C)$ be the space of functions with values in $\C$ that are continuous on the closed disc $\bar D(0;R)$ and holomorphic in its interior. Endowed with the supremum norm, it becomes a separable Banach space.
\begin{lemma}\label{lem:functional_CLT_GAF}
For every $R>0$, weakly on the space $\bA(\bar D(0;2R) \to \C )$, we have
\begin{equation}\label{eq:lem:functional_CLT_GAF}
\left(p_n\left(\frac{z\sin (R/\sqrt n)}{R(1+ \cos(R/\sqrt n))}\right)\right)_{z\in \bar D(0;2R)} \toweak (G(z/2))_{z\in \bar D(0;2R)},
\end{equation}
where $G$ is a Gaussian analytic function as in~\eqref{eq:G GAF def}.
\end{lemma}
\begin{proof}
By the definition of $p_n$, see~\eqref{eq:def_p_n}, we have
$$
\E\left[p_n(z) \overline{p_n(w)}\right] = \sum_{k=0}^n \binom nk z^k \bar w^k = (1+z\bar w)^n
$$
for arbitrary $z,w\in \C$. It follows that
\begin{equation}\label{eq:p_n_covar_z_sqrt_n}
\begin{split}
&\E\left[p_n\left(\frac{z\sin (R/\sqrt n)}{R(1+ \cos(R/\sqrt n))}\right) \overline{p_n\left(\frac{w\sin (R/\sqrt n)}{R(1+ \cos(R/\sqrt n))}\right)}\right]
\\&=
\left(1 +  \frac{z\bar w \sin^2 (R/\sqrt n)}{R^2(1+ \cos(R/\sqrt n))^2} \right)^n
=
\left(1 +  \frac{z\bar w (R^2/n) (1+o(1))}{R^2 (4+o(1))} \right)^n
\ton
\eee^{z\bar w/4}.
\end{split}
\end{equation}
Given that the processes $p_n$ and $G$ are complex Gaussian and centred, this already implies the convergence of the finite-dimensional distributions. To prove that the sequence of processes on the l.h.s.\ of~\eqref{eq:lem:functional_CLT_GAF} is tight on $\bA(\bar D(0;2R) \to \C )$, it suffices to check that
$$
\sup_{z\in \bar D(0;3R)} \sup_{n\in\N} \E\left[\left| p_n\left(\frac{z\sin (R/\sqrt n)}{R(1+ \cos(R/\sqrt n))}\right)\right|^2\right] <\infty;
$$
see~\cite[Remark on p. 341]{shirai}. But this is an easy consequence of~\eqref{eq:p_n_covar_z_sqrt_n}.
\end{proof}
\begin{remark}
Almost the same proof yields that weakly on $\bA(\bar D(0;R) \to \C )$,
$$
\left(p_n\left(\frac z{\sqrt n}\right)\right)_{z\in \bar D(0;R)} \toweak (G(z))_{z\in \bar D(0;R)}.
$$
\end{remark}

\begin{proof}[Proof of Proposition \ref{prop:functional_CLT}]
It follows from Lemma~\ref{lem:functional_CLT_GAF} and~\cite[Theorem~2.8]{billingsley} that the following weak convergence takes place on  the Cartesian square $\bA(\bar D(0;2R) \to \C )^2$:
$$
\left(p_n\left(\frac{z\sin (R/\sqrt n)}{R(1+ \cos(R/\sqrt n))}\right),q_n\left(\frac{z\sin (R/\sqrt n)}{R(1+ \cos(R/\sqrt n))}\right) \right)_{z\in \bar D(0;2R)}
\toweak
(G_1(z/2), G_2(z/2))_{z\in \bar D(0;2R)}.
$$
Consider a map
$$
S: \bA(\bar D(0;2R) \to \C )^2\to \bA(\bar D(0;R) \to \bS^2),
\quad
(f_1,f_2) \mapsto \zeta^{-1} \circ (f_1/f_2).
$$
It is well-defined and continuous outside the closed set
\begin{multline*}
\mathcal Z
:=
\{(f_1,f_2)\in \bA(\bar D(0;2R) \to \C )^2 \colon f_1(z)= f_2(z)=0 \text{ for some } z\in \bar D(0;R)\}
\\
\,
\cup
\,
\{(f_1,0)\colon f_1 \in \bA(\bar D(0;2R) \to \C )\}.
\end{multline*}
For $(f_1,f_2)\in \mathcal Z$ we may define, say, $S(f_1,f_2) =0$. We have (by conditioning on each realization of $G_1$ and applying Fubini's theorem)
$$
\P[(G_1 (\cdot/2),G_2(\cdot/2)) \in \mathcal Z] = 0.
$$
Hence, the Continuous Mapping Theorem~\cite[Theorem~2.7]{billingsley} implies that
$$
\left(\left(\zeta^{-1} \circ \frac {p_n}{q_n}\right)\left(\frac{z\sin (R/\sqrt n)}{R(1+ \cos(R/\sqrt n))}\right)\right)_{z\in \bar D(0;R)} \toweak \left(\zeta^{-1}\left(\frac{G_1(z/2)}{G_2(z/2)}\right)\right)_{z\in \bar D(0;R)}.
$$
This completes the proof after recalling that $\Psi_n = (p_n/q_n)\circ \zeta$.
\end{proof}

\subsection{Local convergence}
Denote by $\nod(\Gamma_n;x_0,R)$ the number of  connected components of $\Gamma_n= \{x\in\bS^2\colon |\Psi_n(x)| = 1\}$ completely contained in the open spherical cap $B(x_0;R)$. Recall also that $\nod(\Gamma;R)$ is the number of connected components of the set $\Gamma = \{z\in\C\colon |\Psi(z)| = 1\}$ contained in the open disc $\{z\in\C\colon |z|<R\}$.
\begin{lemma}\label{lem:expect_local_conv}
For every fixed $R>0$ we have
$$
\lim_{n\to\infty} \E[ \nod(\Gamma_n; x_0, R/\sqrt n)] = \E [\nod (\Gamma; R/2)].
$$
\end{lemma}

For the proof, we need the following two lemmas.  In the case of stationary Gaussian processes, analogous statements were obtained in~\cite[Section~5.3]{nazarov_sodin}.
\begin{lemma}\label{lem:gamma_nonsing}
With probability one, the connected components of $\Gamma$ are non-singular curves.
\end{lemma}

\begin{lemma}\label{lem:non_tangent}
For every $R>0$, the probability that some of the components of $\Gamma$ is tangential to the circle $\{z\in\C\colon |z|=R\}$ is zero.
\end{lemma}

The proof is based on  the following slightly generalized version of the Bulinskaya lemma; see~\cite[Theorem~1]{bulinskaya} or~\cite[Proposition~1.20]{azais_wschebor}. The original proof of Bulinskaya applies with obvious modifications.
\begin{lemma}\label{lem:bulinskaya}
Let $(\xi(t))_{t\in D}$ be a real-valued random field defined on the open set $D\subset  \R^d$. Assume that
\begin{itemize}
\item[(i)] the sample paths of $\xi$ are $C^1$, with probability $1$;
\item[(ii)] for some $\kappa>0$ and every compact set $K\subset D$, the density of $\xi(t)$ exists and is bounded on the interval $(-\kappa, \kappa)$ uniformly over $t\in K$.
\end{itemize}
Then, $\xi$ and its gradient do not have common zeroes, with probability $1$.
\end{lemma}

\begin{proof}[Proof of Lemma~\ref{lem:gamma_nonsing}]
Note that $\Gamma$ is the zero set of the smooth, stationary random function $H(z):=(|G_1(z)|^2 - |G_2(z)|^2)/\eee^{|z|^2}$; see, e.g., Lemma~\ref{lem:cond dist a+bi} for stationarity. Our aim is to show that, with probability $1$, there is no point $z\in \C$ where this function vanishes together with its gradient. By Lemma~\ref{lem:bulinskaya}, it suffices to show that  $H(z)$ has a bounded density. Note that, by stationarity, we may consider $z=0$. The random variables $|G_1(0)|^2$ and $|G_2(0)|^2$ are unit exponential with density $\eee^{-x}\ind_{x>0}$. Hence, the density of $H(0)$ is $\frac 12 \eee^{-|x|}$, which is bounded.
\end{proof}

\begin{proof}[Proof of Lemma~\ref{lem:non_tangent}]
Consider the random function $\xi:[0,2\pi]\to \R$ defined by $\xi(\theta) = H(R\eee^{i\theta})$. It suffices to show that the probability that there is $\theta\in [0,2\pi]$, where this function vanishes together with its derivative, is $0$. By Lemma~\ref{lem:bulinskaya}, it suffices to check that the density of $\xi(\theta)$ is bounded. But $\xi(\theta)$ has the same distribution as $H(0)$ whose density is $\frac 12 \eee^{-|x|}$; see the proof of Lemma~\ref{lem:gamma_nonsing}.
\end{proof}

Now we are in position to prove Lemma~\ref{lem:expect_local_conv}.
\begin{proof}[Proof of Lemma~\ref{lem:expect_local_conv}]
By the spherical invariance of $\Psi_n$, we may assume that $x_0 = (0,0,-1)$ is the south pole of $\bS^2$. First we prove the distributional convergence
\begin{equation}
\label{eq:Nn ball large->N(0;R)}
\nod(\Gamma_n; x_0, R/\sqrt n) \todistr \nod (\Gamma; R/2).
\end{equation}
For a function $f\in \bA(\bar D(0;2R) \to \bS^2)$ let $N(f)$ denote the number of connected components of the set $\{|\zeta\circ f| = 1\}$ lying entirely inside the open disc $D(0;R)$. The map $f\mapsto N(f)$ is well defined and continuous (in fact, locally constant) outside the closed set $\mathcal Y=\mathcal Y_1\cup \mathcal Y_2$, where
\begin{align*}
\mathcal Y_1
&=
\{f\in \bA(\bar D(0;2R) \to \bS^2)\colon \text{ there is $z\in D(0;R)$ such that } |f(z)|=1, \nabla |f(z)|^2 = 0  \},\\
\mathcal Y_2
&=
\{f\in \bA(\bar D(0;2R) \to \bS^2)\colon \text{ there is $z$ such that } |z|=R,   |f(z)|=1,  z^{-1}\nabla |f(z)|^2 \in\R\}.
\end{align*}
In words, $\mathcal Y_1$ is the collection of functions $f$ so that the curve $\Gamma(f):=\{|f(z)|=1\}$ is singular, whereas $\mathcal Y_2$ is a collection of functions $f$ so that $\Gamma(f)$ intersects the boundary of the disc $D(0;R)$ non-transversally. By lemmas~\ref{lem:gamma_nonsing} and~\ref{lem:non_tangent}, we have $\P[\zeta^{-1}\circ \Psi\in \mathcal Y]=0$.
Consider the function
$$
\Phi_n: \bar D(0;2R) \to \bS^2,
\quad
z\mapsto (\zeta^{-1} \circ\Psi_n \circ\zeta^{-1})\left(\frac{z\sin (R/\sqrt n)}{R(1+ \cos(R/\sqrt n))}\right).
$$
Observe that $\nod(\Gamma_n; x_0, R/\sqrt n)$, the number of connected components of the set $$\{x\in \bS^2\colon |\Psi_n(x)|=1\}$$
lying entirely inside the cap $B(x_0;R/\sqrt n)$, is the same as
$N(\Phi_n)$ since the map given in~\eqref{eq:map_cap_disk} is a bijection between $\bar D(0;R)$ and the spherical cap $\bar B(x_0; R/\sqrt n)$. By Proposition~\ref{prop:functional_CLT} together with the continuous mapping theorem~\cite[Theorem~2.7]{billingsley}, we have
$$
\nod(\Gamma_n; x_0, R/\sqrt n) = N(\Phi_n) \todistr N(\Psi(\cdot /2)) = \nod(\Gamma; R/2).
$$


Given the distributional convergence \eqref{eq:Nn ball large->N(0;R)}, to complete the proof of the lemma it remains to show that the family $\{\nod(\Gamma_n; x_0, R/\sqrt n)\}_{n\in\N}$ is uniformly integrable. To this end, we recall that
$$
\nod(\Gamma_n; x_0, R/\sqrt n) = \nod_{\delta/n}(\Gamma_n; x_0, R/\sqrt n) + \nod_{\delta/n - sm}(\Gamma_n; x_0, R/\sqrt n).
$$
As one easily checks, the uniform integrability follows from the following two claims:
\begin{itemize}
\item[(i)] For every $\delta>0$, the random variables $\nod_{\delta/n}(\Gamma_n; x_0, R/\sqrt n)$ are a.s.\ bounded by some constant depending only on $\delta$ and $R$;
\item[(ii)]   $\lim_{\delta\downarrow 0}\limsup_{n\to\infty} \E [\nod_{\delta/n-sm}(\Gamma_n; x_0, R/\sqrt n)] = 0$.
\end{itemize}
Since (ii) follows from Proposition~\ref{prop:sm do small_finite_deg}, we need to verify (i).
To this end, we observe that, by corresponding to $\delta/n$-large
component $\gamma$ lying entirely inside the spherical cap $B(x_0; R/\sqrt n)$ the domain of area $>\delta/n$ enclosed inside
$\gamma$, the number of $\delta/n$-large components of $\Gamma_n$ lying entirely
in $B(x_0; R/\sqrt n)$ is a.s.\
bounded by
\begin{equation*}
\nod_{\delta/n}(\Gamma_n; x_0, R/\sqrt n) \le \frac{\area(B(x_0; R/\sqrt n))}{\delta/n} =  \frac{2\pi (1-\cos(R/\sqrt{n}))}{\delta/n} \le c\cdot \frac{R^{2}}{\delta},
\end{equation*}
with $c>0$ an absolute constant.
Therefore the random variables $\nod_{\delta/n}(\Gamma_n; x_0, R/\sqrt n)$ are uniformly bounded and (ii) holds.
\end{proof}


\subsection{Sandwich estimates}
Our aim is to prove Theorem~\ref{theo:main} stating that
$$
\lim_{n\to\infty} \frac{\E [\nod(\Gamma_n)]}{n} = \pi \cns,
$$
with $\cns$ same as in Theorem \ref{theo:cns_existence}.
For the proof we shall analyze the ``local'' counterparts of $\nod(\Gamma_n)$. Recall that $\nod(\Gamma_n;x,r)$ denotes the number of  connected components of $\Gamma_n$ completely contained in the open cap $B(x;r)$. Similarly, let $\nod^*(\Gamma_n;x,r)$ be the number of connected components of $\Gamma_n$ intersecting the closed cap $\bar B(x;r)$. The next lemma is analogous to the integral-geometric sandwich of Nazarov-Sodin~\cite[Lemma 1 on p.~6]{sodin_lec_notes}. The proof, given below, is a straightforward adaptation of the Nazarov-Sodin's one.

Denote the standard Riemannian volume on $\bS^2$ by $\sigma$. Note that $\sigma(\bS^2) = 4\pi$ and that $\sigma(B(x_0;R))$ does not depend on the choice of the point $x_0\in\bS^2$. We shall frequently use the asymptotic relation
$$
\lim_{r\downarrow 0} \frac{\sigma(B(x_0;r))}{\pi r^2} = 1.
$$

\begin{lemma}\label{lem:sandwich}
For every $r>0$, we have
\begin{equation}\label{eq:sandwich}
\frac{\int_{\bS^2} \nod(\Gamma_n;x,r) \sigma (\dd x)}{\sigma(B(x_0;r))} \leq \nod(\Gamma_n)\leq \frac{\int_{\bS^2} \nod^*(\Gamma_n;x,r) \sigma (\dd x)}{\sigma(B(x_0;r))}.
\end{equation}
\end{lemma}
\begin{proof}
Given  a connected component $\gamma$ of the nodal set $\Gamma_n$ introduce the quantities
\begin{align*}
G(\gamma)
&:=
\bigcap_{v\in\gamma} B(v;r) = \{u\in\bS^2\colon \gamma \subset B(u;r)\},\\
G^*(\gamma)
&:=
\bigcup_{v\in\gamma} \bar B(v;r) = \{u\in\bS^2\colon \gamma \cap \bar B(u;r)\neq \varnothing\}.
\end{align*}
Then, by definition, $\sigma(G(\gamma)) \leq \sigma(B(x_0;r))\leq \sigma(G^*(\gamma))$. Taking the sum over all connected components $\gamma$ of $\Gamma_n$, we arrive at
\begin{equation}
\label{eq:IG sand app}
\sum_{\gamma\subseteq\Gamma_{n}} \sigma(G(\gamma)) \leq \nod(\Gamma_n) \cdot \sigma(B(x_0;r)) \leq \sum_{\gamma\subseteq\Gamma_{n}} \sigma(G^*(\gamma)).
\end{equation}
It remains to observe that the l.h.s.\ of \eqref{eq:IG sand app} equals $\int_{\bS^2} \nod(\Gamma_n;x,r) \sigma(\dd x)$,
whereas the r.h.s.\ of \eqref{eq:IG sand app} is $\int_{\bS^2} \nod^*(\Gamma_n;x,r) \sigma(\dd x)$.
\end{proof}

Taking the expectation in~\eqref{eq:sandwich} and using the Fubini theorem, we arrive at the estimate
$$
\frac{\int_{\bS^2} \E [\nod(\Gamma_n;x,r)] \sigma (\dd x)}{\sigma(B(x_0;r))} \leq \E[ \nod(\Gamma_n)] \leq \frac{\int_{\bS^2} \E [\nod^*(\Gamma_n;x,r)] \sigma (\dd x)}{\sigma(B(x_0;r))}.
$$
Noting that by the $\text{SO}(3)$-invariance of the random rational function $\Psi_n$, the quantities
$\E[ \nod(\Gamma_n; x,r)]$ and $\E [\nod^*(\Gamma_n; x,r)]$ do not depend on the choice of $x\in\bS^2$, we can write
\begin{equation}\label{eq:E_N_n_sandwich}
4\pi \frac{\E [\nod(\Gamma_n; x_0,r)]}{\sigma(B(x_0;r))} \leq \E [\nod(\Gamma_n)] \leq 4\pi \frac{\E [\nod^*(\Gamma_n;x_0,r)]}{\sigma(B(x_0;r))}
\end{equation}
for an arbitrary $x_0\in\bS^2$. These estimates will be important in what follows.

\subsection{Lower bound in Theorem~\ref{theo:main}}
Our aim is to show that
\begin{equation}\label{eq:EN_n_lower_bound}
\liminf_{n\to \infty}\frac{\E [\nod(\Gamma_n)]}{n} \geq \pi \cns.
\end{equation}
Replacing in the first  inequality of~\eqref{eq:E_N_n_sandwich} $r$ by $R/\sqrt n$ and dividing by $n$, we have
$$
\frac{\E [\nod(\Gamma_n)]}{n} \geq 4\pi \frac{\E [\nod(\Gamma_n; x_0,R/\sqrt n)]}{n\sigma(B(x_0;R/\sqrt n))}.
$$
Taking the large $n$ limit, using  Lemma~\ref{lem:expect_local_conv} and observing that $n\sigma(B(x_0;R/\sqrt n)) \to \pi R^2$, we obtain
$$
\liminf_{n\to \infty}\frac{\E [\nod(\Gamma_n)]}{n} \geq 4\frac{\E [\nod(\Gamma;R/2)]}{R^2}
=
\pi \frac{\E [\nod(\Gamma;R/2)]}{\pi (R/2)^2}.
$$
This holds for arbitrary $R>0$. Letting $R\to\infty$ and using Theorem~\ref{theo:cns_existence}, we  arrive at the claimed lower bound~\eqref{eq:EN_n_lower_bound}.

\subsection{Upper bound in Theorem~\ref{theo:main}}
Replacing in the second  inequality of~\eqref{eq:E_N_n_sandwich} $r$ by $(R+1)/\sqrt n$ and dividing by $n$, we have
\begin{equation}\label{eq:E_N_n_upper_proof1}
\frac{\E [\nod(\Gamma_n)]}{n} \leq 4\pi \frac{\E [\nod^*(\Gamma_n;x_0,(R+1)/\sqrt n)]}{n\sigma(B(x_0;(R+1)/\sqrt n))}.
\end{equation}
Let $\nod'(\Gamma_n; x_0, R/\sqrt n, (R+1)/\sqrt n)$ be the number of  connected components of $\Gamma_n$ that intersect the disc $\bar B(x_0; R/\sqrt n)$ and are completely contained inside the larger disc $B(x_0,(R+1)/\sqrt n)$. Further, let $\nod''(\Gamma_n; x_0, R/\sqrt n, (R+1)/\sqrt n)$ be the number of connected components of $\Gamma_n$ that intersect the disc $\bar B(x_0; R/\sqrt n)$ but are not completely contained inside $B(x_0,(R+1)/\sqrt n)$. Evidently,
\begin{equation}
\label{eq:int=lying+not}
\begin{split}
&\E [\nod^*(\Gamma_n; x_0,(R+1)/\sqrt n)]
\\&=
\E [\nod'(\Gamma_n; x_0, R/\sqrt n, (R+1)/\sqrt n)] + \E [\nod''(\Gamma_n; x_0, R/\sqrt n, (R+1)/\sqrt n)].
\end{split}
\end{equation}
Let us provide upper bounds on both expectations on the right-hand side. For the first expectation, we use the trivial estimate
\begin{equation}
\label{eq:bound not lying tot}
\E [\nod'(\Gamma_n; x_0, R/\sqrt n, (R+1)/\sqrt n)] \leq \E [\nod(\Gamma_n; x_0, (R+1)/\sqrt n)].
\end{equation}
To estimate the second expectation, denote by $L_n(A)$ the total spherical length of the nodal set $\Gamma_n$ intersected with an open, connected set $A\subset \bS^2$. Any component contributing to  $\nod''(\Gamma_n; x_0, R/\sqrt n, (R+1)/\sqrt n)$ must have length at least $1/\sqrt n$ inside the ring $B(x_0; (R+1)/\sqrt n )\backslash B(x_0; R/\sqrt n )$, so that we have
\begin{equation}
\label{eq:not lying bnd}
\nod''(\Gamma_n; x_0, R/\sqrt n, (R+1)/\sqrt n) \leq \sqrt n L_n(B(x_0; (R+1)/\sqrt n )\backslash B(x_0; R/\sqrt n )).
\end{equation}
\begin{lemma}\label{lem:exp_length_local}
For every $R>0$ and $x_0\in\bS^2$ we have
\begin{equation}
\label{eq:exp length cap}
\E [L_n(B(x_0; R/\sqrt n ))] = \frac {\pi}{8} \sigma(B(x_0; R/\sqrt n)) \sqrt n.
\end{equation}
\end{lemma}
\begin{proof}
The essential work has already been done by Lerario and Lundberg~\cite{lerario_lundberg} who showed that $\E [L_n(\bS^2)] = (\pi^2/2) \sqrt n$. In fact, their result applies, after scaling, to subdomains of $\bS^2$, which yields~\eqref{eq:exp length cap}, but we also provide an independent argument based only on the above formula for $\E [L_n(\bS^2)]$.
By the rotational invariance, the expectation on the left-hand side of \eqref{eq:exp length cap}
does not depend on $x_0\in\bS^2$. Integrating over $x_0$, we get
$$
4\pi \E [L_n(B(x_0; R/\sqrt n ))] = \int_{\bS^2} \E [L_n(B(x; R/\sqrt n ))]\sigma(\dd x)
=
\E  \left[\int_{\bS^2} L_n(B(x; R/\sqrt n ))\sigma(\dd x)\right].
$$
Fubini's theorem for the Hausdorff length measure implies that
$$
\int_{\bS^2}  L_n(B(x; R/\sqrt n ))\sigma(\dd x)
=
\sigma (B(x_0;R/\sqrt n)) L_n(\bS^2),
$$
from which the claim follows since $\E [L_n(\bS^2)]= (\pi^2/2)\sqrt n$.
\end{proof}
Recall that $n\sigma(B(x_0; R/\sqrt n)) \to \pi R^2$ as $n\to\infty$.  Applying Lemma~\ref{lem:exp_length_local} to the caps $B(x_0; (R+1)/\sqrt n)$ and $B(x_0; R/\sqrt n)$ and subtracting the results, we obtain
\begin{equation}
\label{eq:not lying bnd above}
\lim_{n\to\infty} \sqrt n\, \E [L_n(B(x_0; (R+1)/\sqrt n )\backslash B(x_0; R/\sqrt n ))] = \frac {\pi^2}{8} (2R+1).
\end{equation}
Substituting the bounds \eqref{eq:bound not lying tot} and \eqref{eq:not lying bnd} into \eqref{eq:int=lying+not},
and upon letting $n\to\infty$, while keeping $R$ sufficiently large fixed, we obtain:
\begin{align*}
&\limsup_{n\to\infty} \E [\nod^*(\Gamma_n; x_0,(R+1)/\sqrt n)] \leq
\limsup_{n\to\infty} \E [\nod(\Gamma_n; x_0, (R+1)/\sqrt n)]
\\&+
\limsup_{n\to\infty} \sqrt n \,\E [L_n(B(x_0; (R+1)/\sqrt n )\backslash B(x_0; R/\sqrt n ))]\\
&\quad \quad =
\E \left[\nod \left(\Gamma;\frac{R+1}{2}\right)\right]
+
\frac {\pi^2}{8} (2R+1),
\end{align*}
thanks to \eqref{eq:not lying bnd above}, and lemmas \ref{lem:expect_local_conv} and \ref{lem:exp_length_local}.
Combining this with~\eqref{eq:E_N_n_upper_proof1} and observing that $n\sigma(B(x_0;(R+1)/\sqrt n))\to \pi(R+1)^2$ as $n\to\infty$ yields
\begin{equation*}
\begin{split}
\limsup_{n\to\infty} \frac{\E [\nod(\Gamma_n)]}{n}
&\leq
4\pi \limsup_{n\to\infty} \frac{\E [\nod^*(\Gamma_n; x_0,(R+1)/\sqrt n)]}{n\sigma(B(x_0;(R+1)/\sqrt n))}
\\&\leq
4\pi \frac{\E [\nod (\Gamma;(R+1)/2)]
+
\frac {\pi^2}{8} (2R+1)
}{\pi (R+1)^2}.
\end{split}
\end{equation*}
The above holds for every $R>0$. Letting $R\to\infty$ and applying Theorem~\ref{theo:cns_existence}, we finally arrive at
$$
\limsup_{n\to\infty} \frac{\E [\nod(\Gamma_n)]}{n} \leq \pi \lim_{R\to\infty}\frac{\E [\nod (\Gamma;(R+1)/2)]
}{\pi ((R+1)/2)^2} + 0 = \pi \cns.
$$
This completes the proof of Theorem~\ref{theo:main} assuming the estimates on the number of small components in
propositions~\ref{prop:sm do small}
and~\ref{prop:sm do small_finite_deg}, proved in the course of the next section.

\section{Proof of propositions \ref{prop:sm do small}-\ref{prop:sm do small_finite_deg}:
dismissing the small components}\label{sec:small_components}

\subsection{Small components of \texorpdfstring{$\Psi$}{Psi}: Proof of Proposition \ref{prop:sm do small}}
Recall that
\begin{equation}
\label{eq:Gz sum normals}
G(z) = \sum\limits_{n=0}^{\infty}\frac{\xi_{n}}{\sqrt{n!}}z^{n},
\end{equation}
where $\xi_{n}$ are i.i.d. standard (complex) Gaussian random variables. The function $G(z)$ is GEF with covariance
\begin{equation*}
r_{G}(z,w)=\E\left[G(z)\cdot \overline{G(w)}\right] = \eee^{z\overline{w}},
\end{equation*}
and correlation
\begin{equation*}
\E\left[\frac{G(z)}{\sqrt{r_{G}(z,z)}} \cdot \frac{\overline{G(z)}}{\sqrt{r_{G}(w,w)}}\right] = \eee^{z\overline{w}-|z|^{2}/2-|w|^{2}/2}.
\end{equation*}
We are interested in the number of small components of the function $\Psi$ which has the same distribution as the number of small components of
\begin{equation*}
H^{*}(z) := |G(z)|^{2}-|\widetilde{G}(z)|^{2},
\end{equation*}
where $\widetilde{G}$ is an independent copy of $G$. Equivalently,
we might consider the {\em normalised} random field
\begin{equation}
\label{eq:H(z)=F^2-tild|F|^2}
H(z):=\frac{1}{r_{G}(z,z)}H^{*}(z) = |F(z)|^{2}-|\widetilde{F}(z)|^{2},
\end{equation}
where
\begin{equation}
\label{eq:F=G norm}
F(z):= \frac{1}{\sqrt{r_{G}(z,z)}}G(z),
\end{equation}
and $\widetilde{F}(z)$ its independent copy. We have
\begin{equation}
\label{eq:rF def}
r_{F}(z,w):=\E[F(z)\cdot \overline{F(w)}] = \eee^{z\overline{w}-|z|^{2}/2-|w|^{2}/2}.
\end{equation}

We are interested in the distribution of the random vector
\begin{equation*}
(F(z),\nabla | F(z)|^{2})\in \C \times \R^{2},
\end{equation*}
where we mean $\nabla | F(z)|^{2} := (\partial_{x} | F(z)|^{2},\partial_{y} | F(z)|^{2})$
is the gradient of $z\mapsto |F(z)|^{2}$, considered as a function $\R^{2}\rightarrow\R$.

\begin{lemma}
\label{lem:cond dist a+bi}
Let $z\in \C$ and $a,b\in\R$ be two real numbers.

\begin{enumerate}

\item The distribution of $\nabla |F(0)|^{2}$ conditioned
on $F(0)=a+ib$ is that of two independent centred Gaussians with variance $2(a^{2}+b^{2})$.

\item
\label{it:|F(z)|^2 stat}
The (non-Gaussian) random field $|F(z)|^{2}$ is stationary.

\item
The random field $H(z)$ in \eqref{eq:H(z)=F^2-tild|F|^2} is stationary.

\end{enumerate}

\end{lemma}

\begin{proof}

At first we validate that the first statement of Lemma \ref{lem:cond dist a+bi}.
Conditioned on
\begin{equation}
\label{eq:g=a+ib}
F(0)=G(0)=a+bi=:g
\end{equation}
the Taylor expansion of $G(z)$ around $z=0$ is
$$G(z)=g+\xi_{1}z+o_{z\rightarrow 0}(z);$$ in fact, the conditional law of $(G(z))_{z\in\C}$ is
that of
\begin{equation*}
(G(z))_{z\in\C}  \stackrel{d}{=} \left(g+\sum\limits_{n=1}^{\infty}\frac{\xi_{n}}{\sqrt{n!}} z^{n}\right)_{z\in\C}
\end{equation*}
We then have
\begin{equation*}
|G(z)|^{2} = (g+\xi_{1}z+o(z))\cdot (\overline{g}+\overline{\xi_{1}}\overline{z}+o(z)) = |g|^{2}+2\Re(\overline{g}\xi_{1}z)+o(z),
\end{equation*}
and since $\eee^{|z^{2}|/2} = 1+o(z)$, we also have
\begin{equation}
\label{eq:|F|^2 gxi1z}
|F(z)|^{2} = |g|^{2}+2\Re(\overline{g}\xi_{1}z)+o(z),
\end{equation}
conditioned on $F(0)=g$. Next,
\begin{equation*}
\Re(\overline{g}\xi_{1}z) = \Re(\overline{g}\xi_{1})x-\Im(\overline{g}\xi_{1})y,
\end{equation*}
and, recalling \eqref{eq:g=a+ib}, this, together with \eqref{eq:|F|^2 gxi1z}, yields
\begin{equation}
\label{eq:nabla |F(z)|^2 origin}
\nabla |F(z)|^{2}|_{z=0} = 2\left(\Re(\overline{g}\xi_{1}), -\Im(\overline{g}\xi_{1})\right)
= 2(a\Re\xi_{1}+b\Im\xi_{1}, b\Re\xi_{1}-a\Im\xi_{1})
\end{equation}
That $\nabla |F(z)|^{2}|_{z=0}$ consists of two independent centred Gaussian random variables of variance $2(a^{2}+b^{2})$ now follows
from the fact that $\Re\xi_{1},\Im\xi_{2}$ are independent centred Gaussian random variables of variance $\frac{1}{2}$.

\vspace{2mm}

We now proceed to proving the second statement of Lemma \ref{lem:cond dist a+bi}, i.e.\ that the process $|F(z)|^{2}$ is {\em stationary}. Let us see what is the effect of applying a translation on $F$. For $z_{0}\in \C$ we
have by \eqref{eq:rF def}:
\begin{equation}
\label{eq:rF transl effect}
\begin{split}
r_{F}(z+z_{0},w+z_{0}) &= \eee^{(z+z_{0})\overline{w+z_{0}}-|z+z_{0}|^{2}/2-|w+z_{0}|^{2}/2} =
\eee^{i\Im(z\overline{z_{0}}-
w\overline{z_{0}})}\cdot r_{F}(z,w) \\&=
\E\left[\eee^{i\Im(z\overline{z_{0}})}F(z)\cdot \overline{\eee^{i\Im(w\overline{z_{0}})}F(w)}   \right],
\end{split}
\end{equation}
since
\begin{equation*}
\begin{split}
&(z+z_{0})\overline{w+z_{0}}-|z+z_{0}|^{2}/2-|w+z_{0}|^{2}/2 \\&= z\overline{w}+|z_{0}|^{2}+z\overline{z_{0}}+z_{0}\overline{w}-
|z|^{2}/2-|z_{0}|^{2}/2-\Re(z\overline{z_{0}}) - |w|^{2}/2-|z_{0}|^{2}/2-\Re(w\overline{z_{0}})
\\&=(z\overline{w}-
|z|^{2}/2 - |w|^{2}/2)+(z\overline{z_{0}}-\Re(z\overline{z_{0}}))+(z_{0}\overline{w}-\Re(w\overline{z_{0}}))
\\&= (z\overline{w}-
|z|^{2}/2 - |w|^{2}/2)+i\Im(z\overline{z_{0}})-i\Im(w\overline{z_{0}}).
\end{split}
\end{equation*}
The identity \eqref{eq:rF transl effect} shows that the law of $F(z_{0}+\cdot)$ is equal to the law of
$$\eee^{i\Im(\cdot \overline{z_{0}})}\cdot F(\cdot),$$ and since
for every $z\in \C$ the pre-factor $\eee^{i\Im(\cdot \overline{z_{0}})}$ is of unit absolute value,
it yields that of $|F(\cdot)|^{2}$ is translation invariant,
i.e.\ $|F(\cdot)|^{2}$ is a stationary (though non-Gaussian) process, i.e.\ the $2$nd statement of Lemma \ref{lem:cond dist a+bi}.
The third statement of Lemma \ref{lem:cond dist a+bi} is a straightforward consequence of the $2$nd one via
\eqref{eq:H(z)=F^2-tild|F|^2}.
\end{proof}

\begin{lemma}
\label{lem:F Re Im var}
Let $z=x+yi\in\C$, and
\begin{equation}
\label{eq:F=A+B}
F(z) = A(z)+iB(z)
\end{equation}
be the random field \eqref{eq:F=G norm},
where $A(z)=\Re F(z)$, $B(z) = \Im F(z)$. Then $A(0)$ (resp. $B(0)$) and its derivatives at $z=0$ w.r.t. $x,y$ of all orders
are Gaussian random variables (with finite variance).
\end{lemma}

\begin{proof}

Recall that the covariance function of $F$ is given by \eqref{eq:rF def}; this determines the random field $F$,
and could be used for evaluating all the local expressions like in Lemma \ref{lem:F Re Im var}. We have that
\begin{equation}
\label{eq:rA,rB covar}
r_{A}(z,w):=\E[A(z)\cdot A(w)] = r_{B}(z,w):=\E[B(z)\cdot B(w)] = \frac{1}{2}\Re r_{F}(z)=\frac{1}{2}\Re(\eee^{z\overline{w}-|z|^{2}/2-|w|^{2}/2}),
\end{equation}
From here the finiteness of all the relevant variances is obvious.
\end{proof}

\begin{proof}[Proof of Proposition \ref{prop:sm do small}]

Recall that $H$ is given by \eqref{eq:H(z)=F^2-tild|F|^2}, where $F$ is defined in \eqref{eq:F=G norm} and $\widetilde{F}$ is its independent copy.
Let $\nod_{\delta-sm}(H;R)$ be the number of $\delta$-small nodal components of $H$ lying entirely inside the radius $R$ centred disc $D(0;R)$, i.e.\ those adjacent to domains of area $<\delta$.
Our ultimate goal is proving the estimate \eqref{eq:exp(delta-sm)<=Cdelta^c0R^2 inf}, i.e.\
$$
\E\left[\nod_{\delta-sm}(H;R)\right] \le C\delta^{c_{0}}\cdot R^{2}
$$
for all $R>1$ and $\delta>0$.
There  exist numbers $q>0$, $0<\epsilon<1$, $s>0$, $c_{0} >0$ and a constant $C_{0}>0$
so that the following deterministic inequality holds~\cite{nazarov_sodin,sarnak_wigman} (cf. \eqref{eq:small dom ineq det} below):
\begin{equation}
\label{eq:delta sm bnd det}
\nod_{\delta-sm}(H;R)\le C_{0} \delta^{c_{0}} \left( \int\limits_{D(0;R)} |\partial^{2} H |^{q}dz\right)^{\frac{s}{s+1}} \cdot
\left( \int\limits_{D(0;R)} |H|^{-(1-\epsilon)}\cdot \| \nabla H\|^{-(2-\epsilon)} \dd z\right)^{\frac{1}{s+1}},
\end{equation}
where we denoted $$|\partial^{2} H(z) | :=
\max\limits_{|\alpha|=2}|\partial^{\alpha}H(z)| = \max\left\{|\partial_{x}\partial_{x}H(z)|,
|\partial_{x}\partial_{y}H(z)|,|\partial_{y}\partial_{y}H(z)| \right\}.$$
Taking the expectation of both sides of \eqref{eq:delta sm bnd det} and using the H\"{o}lder inequality, we obtain
\begin{align}
\lefteqn{\E\left[\nod_{\delta-sm}(H;R)\right]}\label{eq:small dom est mom}
\\ &\ll
\delta^{c_{0}}\cdot \left( \E\left[\int\limits_{D(0;R)}|\partial^{2} H |^{q} \dd z\right]\right)^{\frac{s}{s+1}}
\cdot \left( \E\left[\int\limits_{D(0;R)}|H|^{-(1-\epsilon)}\cdot \| \nabla H\|^{-(2-\epsilon)}dz\right]\right)^{\frac{1}{s+1}}\notag
\\&= \delta^{c_{0}}\cdot \left( \int\limits_{D(0;R)}\E\left[|\partial^{2} H(z) |^{q}\right] dz \right)^{\frac{s}{s+1}}
\cdot \left( \int\limits_{D(0;R)}\E\left[|H(z)|^{-(1-\epsilon)}\cdot \| \nabla H(z)\|^{-(2-\epsilon)}\right]dz\right)^{\frac{1}{s+1}}.\notag
\end{align}

First, by the stationarity of $H(z)$ (Lemma \ref{lem:cond dist a+bi}, part $3$), both integrands
\begin{equation}
\label{eq:int pos mom gen}
\E\left[|\partial^{2} H(z) |^{q}\right] \equiv \E\left[|\partial^{2} H(0) |^{q}\right]
\end{equation}
and
\begin{equation}
\label{eq:int neg mom gen}
\E\left[|H(z)|^{-(1-\epsilon)}\cdot  \| \nabla H(z)\|^{-(2-\epsilon)}\right] \equiv \E\left[|H(0)|^{-(1-\epsilon)} \cdot \| \nabla H(0)\|^{-(2-\epsilon)}\right]
\end{equation}
are constant, independent of $z$. The estimate \eqref{eq:exp(delta-sm)<=Cdelta^c0R^2 inf} will follow
from \eqref{eq:small dom est mom} once we establish that both
\begin{equation}
\label{eq:int pos mom}
\E\left[|\partial^{2} H(0) |^{q}\right] < \infty
\end{equation}
and
\begin{equation}
\label{eq:int neg mom}
\E\left[|H(0)|^{-(1-\epsilon)} \cdot \| \nabla H(0)\|^{-(2-\epsilon)}\right] < \infty
\end{equation}
are {\em finite}.

First, to establish \eqref{eq:int pos mom}, by \eqref{eq:H(z)=F^2-tild|F|^2} the l.h.s. of \eqref{eq:int pos mom} involves (positive)
moments of the real and imaginary parts of $F,\widetilde{F}$ and its couple of derivatives at the origin,
and these are finite by Lemma \ref{lem:F Re Im var}.
Concerning \eqref{eq:int neg mom}, we notice that, unlike
the Gaussian stationary case that were treated in the earlier manuscripts, $H(0)$ and $\nabla H(0)$ (or, more generally,
$H(z)$ and $\nabla H(z)$) are {\em not independent},
and so the moments of the r.h.s.\ of \eqref{eq:int neg mom} do not split nicely as the did in other cases.
Instead we are going to condition on the values of $F(0)$ and $\widetilde{F}(0)$
(and hence of $H(0)$ via \eqref{eq:H(z)=F^2-tild|F|^2}), and invoke Lemma \ref{lem:cond dist a+bi}.

We write
\begin{equation*}
\nabla H (0) = \nabla (|F(0)|^{2})-\nabla (|\widetilde{F}(0)|^{2}),
\end{equation*}
and take $v=(a,b)\in\R^{2}$, $\widetilde{v}=(\widetilde{a},\widetilde{b})\in \R^{2}$, and evaluate
$\E\left[|H(0)|^{-(1-\epsilon)} \| \nabla H(0)\|^{-(2-\epsilon)}\right]$ by first conditioning on
the values $F(0) = a+bi$, $\widetilde{F}(0)=\widetilde{a}+\widetilde{b}i$, and use Lemma \ref{lem:cond dist a+bi}
to infer the conditional distribution of $\nabla H(0)$ to be centred Gaussian with independent components, each having variance
$$2(a^{2}+b^{2}+\widetilde{a}^{2}+\widetilde{b}^{2})=2(\|v\|^{2}+\|\widetilde{v}\|^{2}).$$
As a result, we have
\begin{equation*}
\begin{split}
&\E\left[|H(0)|^{-(1-\epsilon)} \cdot\| \nabla H(0)\|^{-(2-\epsilon)}\right] =
\iint\limits_{\R^{2}\times \R^{2}}\frac{1}{\left|\|v\|^{2}-\|\widetilde{v}\|^{2}\right|^{1-\epsilon}}
\cdot
\frac{\exp(-(\|v\|^{2}+\|\widetilde{v}\|^{2}))dvd\widetilde{v}}{\pi^{2}} \times \\&\times
\int\limits_{\R^{2}}
\frac{1}{\left(\|v\|^{2}+\|\widetilde{v}\|^{2}\right)^{1-\epsilon/2}\cdot 2^{1-\epsilon/2}
\|\eta\|^{2-\epsilon}}
\frac{\exp(-\|\eta\|^{2}/2)d\eta}{(2\pi)},
\end{split}
\end{equation*}
with the natural scaling $$\eta=(\eta_{1},\eta_{2}):=
\frac{1}{\sqrt{2}(\|v\|^{2}+\|\widetilde{v}\|^{2})^{1/2}}\nabla H(z)  ,$$
so that $(\eta_{1},\eta_{2})$ are standard i.i.d. Gaussian.

Continuing, we have
\begin{equation}
\label{eq:int neg mom red}
\begin{split}
&\E\left[|H(0)|^{1-\epsilon} \| \cdot\nabla H(0)\|^{-(2-\epsilon)}\right] \ll
\iint\limits_{\R^{2}\times \R^{2}}\frac{1}
{\left|\|v\|^{2}-\|\widetilde{v}\|^{2}\right|^{1-\epsilon} \cdot \left(\|v\|^{2}+\|\widetilde{v}\|^{2}\right)^{1-\epsilon/2}}
\times \\&\times
\exp(-(\|v\|^{2}+\|\widetilde{v}\|^{2}))dvd\widetilde{v}
\int\limits_{\R^{2}}
\frac{\exp(-\|\eta\|^{2}/2)}{\|\eta\|^{2-\epsilon}}
d\eta
\\& =
\iint\limits_{\R^{2}\times \R^{2}}\frac{\left|\|v\|^{2}-\|\widetilde{v}\|^{2}\right|^{\epsilon/2}}
{\left|\|v\|^{4}-\|\widetilde{v}\|^{4}\right|^{1-\epsilon/2}}
\exp(-(\|v\|^{2}+\|\widetilde{v}\|^{2}))dvd\widetilde{v} \cdot
\int\limits_{\R^{2}}
\frac{\exp(-\|\eta\|^{2}/2)}{\|\eta\|^{2-\epsilon}}
d\eta.
\end{split}
\end{equation}
The finiteness of the expectation \eqref{eq:int neg mom} will follow from \eqref{eq:int neg mom red} once we show that the integral
\begin{equation*}
\iint\limits_{\R^{2}\times \R^{2}}\frac{\left|\|v\|^{2}-\|\widetilde{v}\|^{2}\right|^{\epsilon/2}}
{\left|\|v\|^{4}-\|\widetilde{v}\|^{4}\right|^{1-\epsilon/2}}
\exp(-(\|v\|^{2}+\|\widetilde{v}\|^{2}))dvd\widetilde{v} < \infty
\end{equation*}
is finite, which, in turn, would follow from the convergence of the integral
\begin{equation}
\label{eq:int 1/v^4-tild v^4}
\iint\limits_{\R^{2}\times \R^{2}}\frac{dvd\widetilde{v}}
{\left|\|v\|^{4}-\|\widetilde{v}\|^{4}\right|^{1-\epsilon/2}} \exp\left(-\frac{1}{2}(\|v\|^{2}+\|\widetilde{v}\|^{2})\right)
 < \infty.
\end{equation}
To show that the integral on the r.h.s.\ of \eqref{eq:int 1/v^4-tild v^4} is convergent, we use the slightly
non-standard spherical coordinates
\begin{equation}
\label{eq:spher sep coord}
\begin{split}
(\rho,\theta,\phi_{1},\phi_{2}) \mapsto (v,\widetilde{v})=\big((&\rho\cos(\theta)\cos(\phi_{1}),\rho\cos(\theta)\sin(\phi_{1})),
\\&(\rho\sin(\theta)\cos(\phi_{2}),\rho\sin(\theta)\sin(\phi_{2}))\big),
\end{split}
\end{equation}
whose Jacobian is
\begin{equation*}
J^{*}(\rho,\theta,\phi_{1},\phi_{2}) = \rho^{3}\cdot J(\theta,\phi_{1},\phi_{2}),
\end{equation*}
for some explicit function $J$; $J$ continuous, and, in particular, bounded on $[0,\pi/2]\times [0,2\pi]^{2}$.
Using that, by \eqref{eq:spher sep coord}, we have $\|v\|^{2} = \rho^{2}\cos(\theta)^{2}$, $\|\widetilde{v}\|^{2} = \rho^{2}\sin(\theta)^{2}$,
one may rewrite the integral on the r.h.s.\ of \eqref{eq:int 1/v^4-tild v^4} as
\begin{equation*}
\begin{split}
&\iint\limits_{\R^{2}\times \R^{2}}\frac{dvd\widetilde{v}}
{\left|\|v\|^{4}-\|\widetilde{v}\|^{4}\right|^{1-\epsilon/2}} \exp\left(-\frac{1}{2}(\|v\|^{2}+\|\widetilde{v}\|^{2})\right)
\\&= \iint\limits_{(0,\infty)\times [0,\pi/2]\times [0,2\pi]^{2}}\frac{\rho^{3}|J(\theta,\phi_{1},\phi_{2})|d\theta d\phi_{1}d\phi_{2}}
{\rho^{4-2\epsilon}\left|\cos(\theta)^{4}-\sin(\theta)^{4}\right|^{1-\epsilon/2}} \exp(-\rho^{2}/2)d\rho
\\&= \int\limits_{0}^{\infty}\frac{1}{\rho^{1-2\epsilon}}\eee^{-\rho^{2}/2}d\rho \times \int\limits_{[0,\pi/2]\times [0,2\pi]^{2}}
\frac{|J(\theta,\phi_{1},\phi_{2})|d\theta d\phi_{1}d\phi_{2}}{|\cos(2\theta)|^{1-\epsilon/2}} < \infty,
\end{split}
\end{equation*}
is finite, by Taylor expanding $$\cos(2\theta)\gg \theta-\frac{\pi}{4}$$ around the pole $\theta=\pi/4$ in the relevant range;
that is, \eqref{eq:int neg mom} is now established.
Consolidating all of our arguments, we insert the newly proven \eqref{eq:int neg mom} teamed with the earlier estimate \eqref{eq:int pos mom},
upon bearing in mind \eqref{eq:int neg mom gen} and \eqref{eq:int pos mom gen} respectively, into \eqref{eq:small dom est mom},
finally implies the statement \eqref{eq:exp(delta-sm)<=Cdelta^c0R^2 inf} of Proposition \ref{prop:sm do small}.
\end{proof}

\subsection{Small components of \texorpdfstring{$\Psi_n$}{Psin}: Proof of Proposition~\ref{prop:sm do small_finite_deg}}
We start by recalling the necessary notation.  Let $p_{n}(z),{q}_{n}(z):\C\rightarrow\C$ be the two independent complex valued random polynomials as in~\eqref{eq:def_p_n}. Their covariance is given by
\begin{equation*}
r_{p_{n}}(z,w) = (1+z\overline{w})^{n}, \qquad z,w\in\C.
\end{equation*}
Further, let $P_{n},Q_{n}:\bS^{2}\rightarrow\C\cup\{\infty\}$ be two independent complex valued,
random Gaussian fields on the unit sphere $\bS^2$ defined by $P_{n}(x) = p_{n}(\zeta(x))$, $Q_{n}(x)=p_{n}(\zeta(x))$, where $x=(u,v,w)\in\bS^{2}$,
\begin{equation}
\label{eq:zeta stereo def}
\zeta(x) = \frac{u+iv}{1-w}\in \C\cup \{\infty\},
\end{equation}
is the stereographic projection. The covariance function of $P_{n}$ is
\begin{equation*}
r_{P_{n}}(x,y) = r_{p_{n}}(\zeta(x),\zeta(y)) = \left(1+\zeta(x)\overline{\zeta(y)}\right)^{n}.
\end{equation*}
We record here that in the spherical coordinates $(\theta,\phi)$ for $x\in \bS^2$,
the stereographic projection $\zeta$ in \eqref{eq:zeta ster proj def} is given by
\begin{equation}
\label{eq:zeta stereo spher coord}
\zeta(x) =\frac{1}{\tan(\theta/2)}\eee^{i\phi}.
\end{equation}
The central object of our study is the (non-Gaussian) ensemble $\{\Psi_{n}\}_{n\ge 1}$ of complex-valued functions
\begin{equation}
\label{eq:Psin def1}
\Psi_{n}(x) = \frac{P_{n}(x)}{Q_{n}(x)},
\end{equation}
defined on the sphere
$\Psi_{n}:\bS^{2}\rightarrow\C$.

\begin{lemma}
For every $n\ge 1$, the random field $\Psi_{n}$ as in \eqref{eq:Psin def1} is rotation invariant, i.e.\
for every $g\in SO(3)$, we have
\begin{equation*}
(\Psi_{n}(g x))_{x\in \bS^2} \eqdistr (\Psi_{n}(x))_{x\in \bS^2}.
\end{equation*}
\end{lemma}

\begin{proof}
First let us consider the effect of a rotation $g\in SO(3)$ on $P_{n}$.
It is known that
\begin{equation*}
\zeta g \zeta^{-1}(z) = \frac{\lambda z +\mu}{-\overline{\mu}z+\overline{\lambda}}
\end{equation*}
for some $\mu,\lambda\in\C$ with  $|\lambda|^{2}+|\mu|^{2}=1$.
Letting $z=\zeta (x)$, $w=\zeta (y)$, we have
\begin{equation*}
r_{P_{n}}(gx,gy) = r_{P_{n}}(g\zeta^{-1}(z),g\zeta^{-1}(w)) =  r_{p_{n}}(\zeta g\zeta^{-1}(z),\zeta g\zeta^{-1}(w))
= \left(1+ \zeta g\zeta^{-1}(z) \cdot \overline{\zeta g\zeta^{-1}(w)}  \right)^{n}.
\end{equation*}
Now,
\begin{equation*}
\begin{split}
1+ \zeta g\zeta^{-1}(z) \cdot \overline{\zeta g\zeta^{-1}(w)} &=
1+\frac{\lambda z +\mu}{-\overline{\mu}z+\overline{\lambda}} \cdot
\frac{\overline{\lambda}\overline{w} +\overline{\mu}}{-\mu\overline{w}+\lambda}
= 1+\frac{|\lambda |^{2}z\overline{w}+\lambda\overline{\mu}z + \overline{\lambda}\mu\overline{w}+|\mu|^{2}}{(-\overline{\mu}z+\overline{\lambda})\cdot (-\mu\overline{w}+\lambda)}
\\&=1+\frac{(1-|\mu |^{2})z\overline{w}+\lambda\overline{\mu}z + \overline{\lambda}\mu\overline{w}+(1-|\lambda|^{2})}{(-\overline{\mu}z+\overline{\lambda})\cdot (-\mu\overline{w}+\lambda)}
\\&= 1+\frac{(1+z\overline{w})-\left(|\mu |^{2}z\overline{w}-\lambda\overline{\mu}z - \overline{\lambda}\mu\overline{w}+|\lambda|^{2}\right)}{(-\overline{\mu}z+\overline{\lambda})\cdot (-\mu\overline{w}+\lambda)}
\\&= \frac{1+z\overline{w}}{(-\overline{\mu}z+\overline{\lambda})\cdot (-\mu\overline{w}+\lambda)}.
\end{split}
\end{equation*}
Therefore,
\begin{equation}
\label{eq:rPn effect rot}
r_{P_{n}}(gx,gy) = \frac{r_{P_{n}}(x,y)}{(-\overline{\mu}z+\overline{\lambda})^{n}\cdot (-\mu\overline{w}+\lambda)^{n}}
\end{equation}
or, equivalently,
\begin{equation*}
(P_{n}(gx))_{x\in\bS^2} \eqdistr \left(\frac{P_{n}(x)}{(-\overline{\mu}\zeta (x)+\overline{\lambda})^{n}}\right)_{x\in\bS^2},
\end{equation*}
so that
\begin{equation*}
\Psi_{n}(gx) = \frac{P_{n}(gx)}{Q_{n}(gx )} \eqdistr \frac{P_{n}(x)}{(-\overline{\mu}\zeta (x)+\overline{\lambda})^{n}} \cdot
\frac{(-\overline{\mu}\zeta (x)+\overline{\lambda})^{n}}{Q_{n}(x)} = \frac{P_{n}(x)}{Q_{n}(x)} = \Psi_{n}(x).
\end{equation*}
\end{proof}

The lemniscate $\Gamma_n = \{ x\in\bS^{2}: |\Psi_{n}(x)| = 1\}$
is equal to the zero set (nodal line) of
\begin{equation}
\label{eq:Hn* def}
H_{n}^{*}(x):= |P_{n}(x)|^{2}- |Q_{n}(x)|^{2}.
\end{equation}
In turn, the nodal line of $H_{n}^{*}(x)$ is equal to the nodal line of
the {\em normalised} random field
\begin{equation}
\label{eq:Hn norm def}
H_{n}(x):= \frac{1}{r_{P_{n}}(x,x)}H_{n}^{*}(x) = |F_{n}(x)|^{2}-|\widetilde{F_{n}}(x)|^{2} ,
\end{equation}
where
\begin{equation}
\label{eq:Fn norm def}
F_{n}(x) = \frac{P_{n}(x)}{\sqrt{r_{P_{n}}(x,x)}}
\end{equation}
is unit variance and $\widetilde{F_{n}}$ is its independent copy, so that
\begin{equation*}
r_{F_{n}}(x,y)=r_{\widetilde{F_{n}}}(x,y) = \frac{r_{P_{n}}(x,y)}{\sqrt{r_{P_{n}}(x,x)\cdot r_{P_{n}}(y,y)}}
= \frac{(1+\zeta(x)\cdot\overline{\zeta(y)})^{n}}{(1+|\zeta(x)|^{2})^{n/2}\cdot (1+|\zeta(y)|^{2})^{n/2}}.
\end{equation*}
Hence bounding the number of small components of the lemniscate $\Gamma_n$ is equivalent to bounding the number of
small nodal components of $H_{n}$.

\begin{lemma}

\begin{enumerate}

\item

The law
of $|F_{n}|^{2}$, with $F_{n}$ as in \eqref{eq:Fn norm def}, is rotation invariant, i.e.\
for every $g\in SO(3)$, we have
\begin{equation}
\label{eq:|Fn|^2 invar}
(|F_{n}(gx)|^{2})_{x\in \bS^2} \eqdistr (|F_{n}(x)|^{2})_{x\in\bS^2}.
\end{equation}

\item
The law of $H_{n}$ as in \eqref{eq:Hn norm def}, is s rotation invariant, i.e.\
for every $g\in SO(3)$, we have
\begin{equation}
\label{eq:Hn invar}
(H_{n}(g x))_{x\in \bS^2} \eqdistr (H_{n}(x))_{x\in \bS^2}.
\end{equation}

\end{enumerate}

\end{lemma}

\begin{proof}

It is evident that \eqref{eq:Hn invar} is a straightforward consequence of \eqref{eq:|Fn|^2 invar} via \eqref{eq:Hn norm def}
(and the independence of $F_{n}$ and $\widetilde{F_{n}}$), so in what follows
we only restrict ourselves to proving \eqref{eq:|Fn|^2 invar}.
To this end let us check what effect rotating by $g$ has on the law of $F_{n}$.
Letting $z=\zeta(x)$, $w=\zeta(y)$, and using \eqref{eq:rPn effect rot} once again, we have
\begin{equation*}
\begin{split}
r_{F_{n}}(gx,gy) &= \frac{r_{P_{n}}(gx,gy)}{\sqrt{r_{P_{n}}(gx,gx)\cdot r_{P_{n}}(gy,gy)}} \\&=
\frac{r_{P_{n}}(x,y)}{(-\overline{\mu}z+\overline{\lambda})^{n}\cdot (-\mu\overline{w}+\lambda)^{n}}  \cdot
\frac{|\overline{\lambda}-\overline{\mu}z|^{n}\cdot |\overline{\lambda}-\overline{\mu}w|^{n}}{\sqrt{r_{P_{n}}(x,x)\cdot r_{P_{n}}(y,y)}}
\\&= \frac{r_{P_{n}}(x,y)}{\sqrt{r_{P_{n}}(x,x)\cdot r_{P_{n}}(y,y)}} \cdot \frac{|\overline{\lambda}-\overline{\mu}z|^{n}\cdot |\overline{\lambda}-\overline{\mu}w|^{n}}{(-\overline{\mu}z+\overline{\lambda})^{n}\cdot (-\mu\overline{w}+\lambda)^{n}} \\&=
\frac{|\overline{\lambda}-\overline{\mu}z|^{n}\cdot |\overline{\lambda}-\overline{\mu}w|^{n}}{(-\overline{\mu}z+\overline{\lambda})^{n}\cdot (-\mu\overline{w}+\lambda)^{n}}\cdot r_{F_{n}}(x,y).
\end{split}
\end{equation*}
Noting that the pre-factor $$\frac{|\overline{\lambda}-\overline{\mu}z|^{n}\cdot |\overline{\lambda}-\overline{\mu}w|^{n}}{(-\overline{\mu}z+\overline{\lambda})\cdot (-\mu\overline{w}+\lambda)}$$
is of unit absolute value, it follows that the non-Gaussian random field $|F_{n}(\cdot)|^{2}$ is invariant w.r.t. $g\in SO(3)$.
\end{proof}

For a smooth function $g:\bS^2\rightarrow\R$ denote by $\nod_{\delta-sm}(g; x_0,r)$ the number of $\delta$-small nodal
components of $g$ lying inside the spherical cap $B(x_0;r)$, i.e.\ those adjacent to domains of spherical area $<\delta$. We have the following bound for small components.

\begin{lemma}[\cite{nazarov_sodin,sarnak_wigman}]
There exist numbers $q>2$, $0<\epsilon<1$,
\begin{equation}
\label{eq:s val scal}
s=\frac{3-2\epsilon}{q} >0,
\end{equation}
\begin{equation}
\label{eq:t val scal}
t =(1-\epsilon)\left(1-\frac{1}{q}\right)+(2-\epsilon)\left(\frac{1}{2}-\frac{1}{q} \right) - 1,
\end{equation}
\begin{equation}
\label{eq:c0 exp val scal}
c_{0} = \frac{t}{s+1}
\end{equation}
and $C_{0}>0$
so that for any smooth function $g:\bS^{2}\rightarrow\R$,
any $x_0\in \bS^2$ and $r>0$, the following inequality holds:
\begin{equation}
\label{eq:small dom ineq det}
\begin{split}
\nod_{\delta-sm}(g; x_0,r) &\le C_{0}\delta^{c_{0}} \left(\int\limits_{B(x_0;r)}|\partial^{2}H_{n}|^{q}dz\right)^{\frac{s}{s+1}}
\times\\&\times \left(\int\limits_{B(x_0;r)} | H_{n} |^{-(1-\epsilon)} \cdot \|\nabla H_{n}   \|^{-(2-\epsilon)}  dz \right)^{\frac{1}{s+1}}.
\end{split}
\end{equation}
Here we denoted
\begin{equation*}
|\partial^{2}H_{n}(x)| := \max\limits_{V_{1},V_{2}\in T_{x}(\bS^{2})} |V_{1}V_{2}H_{n}(x)|.
\end{equation*}
\end{lemma}

The following result is, in light of
\eqref{eq:Hn norm def}, \eqref{eq:Hn* def}, and \eqref{eq:Psin def1}, a straightforward restatement of
Proposition \ref{prop:sm do small_finite_deg}.

\begin{proposition}
\label{prop:small comp est scal}
There exists a number $C_{1}>0$, so that the expected number of $\delta/n$-small components of $H_{n}$ as in
\eqref{eq:Hn norm def}, with $n$ sufficiently large and
$\delta>0$ sufficiently small, is bounded by
\begin{equation}
\label{eq:small dom ineq det1}
\E\left[\nod_{\delta/n-sm}(H_{n}; x_0, R/\sqrt{n})\right] \le C_{1}\delta^{c_{0}} \cdot R^2
\end{equation}
uniformly for all $x_0\in \bS^2$ and $1<R<\sqrt{n}$, with $c_{0}$ same as in \eqref{eq:c0 exp val scal}.
\end{proposition}

Following along the same lines as in the proof below, except for the integrals taken over the whole of $\bS^2$ in place of $B(x_{0};R/\sqrt{n})$, it is possible to prove the following global version of~\eqref{eq:small dom ineq det1}:
\begin{equation}
\label{eq:small dom ineq det1 glob}
\E[\nod_{\delta/n-sm}(H_{n})] \le C_{1}\delta^{c_{0}} \cdot n.
\end{equation}

\begin{proof}
Taking the expectation of both sides of \eqref{eq:small dom ineq det}, putting $r=R/\sqrt n$ and using the H\"older inequality yields:
\begin{equation}
\label{eq:small dom ineq rand}
\begin{split}
\E\left[\nod_{\delta-sm}(H_{n}; x_0,R/\sqrt n)\right]
&\le C_{0}\delta^{c_{0}} \left(\int\limits_{B(x_0;R/\sqrt n)}\E\left[|\partial^{2}H_{n}|^{q}\right]dz\right)^{\frac{s}{s+1}}\times
\\
&\times \left(\int\limits_{B(x_0;R/\sqrt n)} \E\left[| H_{n} |^{-(1-\epsilon)} \|\nabla H_{n}   \|^{-(2-\epsilon)}\right]  dz \right)^{\frac{1}{s+1}}.
\end{split}
\end{equation}
By \eqref{eq:small dom ineq rand} and the rotational invariance of $H_{n}$ as above, it is sufficient to bound each of
the two expectations on the r.h.s.\ of \eqref{eq:small dom ineq rand} at an arbitrary fixed point, e.g.\ the South Pole $x_{0}=S= (0,0,-1)$;
we claim the (finite) bounds
\begin{equation}
\label{eq:pos mom 2nd der bnd scal}
\E\left[|\partial^{2}H_{n}|^{q}\right]dz \ll_{q} n^{q},
\end{equation}
and
\begin{equation}
\label{eq:neg mom 1st mix der bnd scal}
\E\left[| H_{n} |^{-(1-\epsilon)} \|\nabla H_{n}   \|^{-(2-\epsilon)}\right] \ll n^{-1+\epsilon/2}.
\end{equation}
Once \eqref{eq:pos mom 2nd der bnd scal} and \eqref{eq:neg mom 1st mix der bnd scal}
are established, we substitute these into \eqref{eq:small dom ineq rand} with $\delta$ replaced by $\delta/n$,
bearing in mind the rotation invariance, to yield
\begin{equation*}
\E\left[\nod_{\delta/n-sm}(H_{n}; x_0,R/\sqrt n)\right] \ll \delta^{c_{0}} \cdot n^{-c_{0}+\frac{qs}{s+1}-\frac{1-\epsilon/2}{s+1}} \cdot R^2/n
= \delta^{c_{0}} \cdot R^2,
\end{equation*}
i.e. the bound in \eqref{eq:small dom ineq det1} of Proposition \ref{prop:small comp est scal},
as, by our choice of the parameters \eqref{eq:s val scal}, \eqref{eq:t val scal} and \eqref{eq:c0 exp val scal}, we have
\begin{equation*}
-c_{0}+\frac{qs}{s+1}-\frac{1-\epsilon/2}{s+1}=1.
\end{equation*}

First, we turn to proving \eqref{eq:neg mom 1st mix der bnd scal}. It is convenient to evaluate the relevant expressions
for the South Pole $x_{0}=S=(0,0,-1)$, mapping under the stereographic map \eqref{eq:zeta stereo def} to the origin $0\in\C$. However,
as we are going to invoke the spherical coordinates, with $\zeta$ given by \eqref{eq:zeta stereo spher coord},
this is not optimal, as the corresponding orthonormal frame blows up there, posing a difficulty evaluating the gradient
on the l.h.s. of \eqref{eq:neg mom 1st mix der bnd scal}. One way to resolve this is by expressing $$p_{n}(z)=p_{n}(\zeta(x))=P_{n}(x)$$
around $z=0$ on the Euclidean plane, and then, noting that the tangent plane $T_{S}(\bS^{2})$ of the sphere at $S$ is
canonically isomorphic to the plane $\pi_{1}:\{z=-1\}\subseteq\R^{3}$
(so also canonically isomorphic to $\pi_{2}\{z=0\}\cong\R^{2}\cong\C$), use that, as the distances on
$\pi_{1}$ are magnified by a factor of $2$ around $S$ relatively to $\pi_{2}$, we should have
\begin{equation}
\label{eq:grad spher=1/2 grad Eucl}
\nabla_{\sph} |P_{n}(\cdot)|^{2}|_{S} = \frac{1}{2}\nabla_{\text{Eucl}} |p_{n}(\cdot)|^{2}|_{0},
\end{equation}
where we denoted $\nabla_{\sph}$ to be the spherical gradient, and $\nabla_{\text{Eucl}}$ to be the gradient on $\R^{2}$.

However, proving \eqref{eq:grad spher=1/2 grad Eucl} by an explicit computation seems technically demanding.
Instead we are going to study the distribution of
the relevant random variables {\em around} $S$, and then
use the intrinsic definition of the l.h.s. of \eqref{eq:neg mom 1st mix der bnd scal} to infer its law {\em at} the South Pole.
We are interested in the distribution of $\nabla H_{n}(S)$ conditioned on $F_{n}(S)=g=a+ib$ and $\widetilde{F_{n}}(S)=\widetilde{g}=\widetilde{a}+i\widetilde{b}$. Let us consider first the (rotation invariant)
distribution of $\nabla |F_{n}(x_{0})|^{2} := \nabla |F_{n}(x)|^{2} |_{x=x_{0}}$ understood as the gradient on the sphere
of the real valued function $|F_{n}(x)|^{2}:\bS^{2}\rightarrow\R$ conditioned on $F_{n}(x_{0}) = g=a+ib$,
where $x_{0}\in\bS^{2}$ is a point in the neighbourhood of $S$.
First,
\begin{equation}
\label{eq:rPn(Z0,Z0)->1}
r_{P_{n}}(x_{0},x_{0}) = r_{p_{n}}(\zeta(x_{0}),\zeta(x_{0})) = (1 + |\zeta(x_{0})|)^{n},
\end{equation}
so that for $x_{0}$ sufficiently close to $S$, conditioning on $F_{n}(S)=g$ is asymptotic to conditioning on $P_{n}(x_{0})=g$.
Now, let $x=x(\theta,\phi)$ and
$z = \zeta(x) = \zeta(x(\theta,\phi))$
\begin{equation}
\label{eq:pn def}
P_{n}(x) = p_{n}(\zeta(x)) = \sum\limits_{k=0}^{n} \xi_{k}\sqrt{{\binom n k}} z^{k},
\end{equation}
so, conditioned on $p_{n}(\zeta(x_{0}))=P_{n}(x_{0})=g$,
\begin{equation*}
P_{n}(x) = p_{n}(\zeta(x)) = g+\sqrt{n}\widehat{\xi_{1}}(z-z_{0})+  \sum\limits_{k=2}^{n} \widehat{\xi_{n}}\sqrt{{\binom n k}} (z-z_{0})^{k},
\end{equation*}
with all $\{\widehat{\xi_{i}}\}_{1\le i \le n}$ Gaussian, asymptotic to $\xi_{i}$ in \eqref{eq:pn def}, i.e.\ standard Gaussian i.i.d. $z_{0}\rightarrow 0$ (equivalent to $x_{0}\rightarrow S$.
Now
\begin{equation*}
\begin{split}
|P_{n}(x)|^{2} &= \left(g+\sqrt{n}\widehat{\xi_{1}}(z-z_{0})+  \sum\limits_{k=2}^{n} \widehat{\xi_{n}}\sqrt{{\binom n k}} (z-z_{0})^{k}\right)
 \times \\&\times\left(\overline{g}+\sqrt{n}\overline{\widehat{\xi_{1}}}\overline{(z-z_{0})}+  \sum\limits_{k=2}^{n} \sqrt{{\binom n  k}} \cdot \overline{\widehat{\xi_{n}}} \overline{(z-z_{0})}^{k}\right)
\\&= |g|^{2} + 2\sqrt{n}\cdot  \Re(g\overline{\widehat{\xi_{1}}}\overline{(z-z_{0})})+   n|\widehat{\xi_{1}}|^{2}|(z-z_{0})|^{2}\\&+
\sum\limits_{k,m\ge 2} \sqrt{{\binom n k}} \sqrt{\binom  n m} \widehat{\xi_{k}}\overline{\widehat{\xi_{m}}} (z-z_{0})^{k}
\overline{(z-z_{0})}^{m},
\end{split}
\end{equation*}
\begin{equation*}
|F_{n}(x)|^{2} = \frac{|P_{n}(x)|^{2}}{r_{P_{n}}(x,x)} = \frac{|P_{n}(x)|^{2}}{(1+|z|^{2})^{n}}
\end{equation*}
so that, bearing in mind that
\begin{equation*}
z^{k}\overline{z}^{m} = \frac{1}{\tan(\theta/2)^{k+m}}\eee^{i(k-m)\phi},
\end{equation*}
so that
\begin{equation*}
\partial_{\theta}z^{k}\overline{z}^{m}|_{x_{0}} = -\frac{1}{2}(k+m)\frac{1}{\tan(\theta/2)^{k+m+1}\cos(\theta/2)^{2}}\eee^{i(k-m)\phi} \rightarrow 0
\end{equation*}
at $x_{0}\rightarrow S$, provided that $k+m \ge 2$,
and
\begin{equation*}
\frac{1}{\sin{\theta}}\partial_{\phi}z^{k}\overline{z}^{m} =
i(k-m)\frac{1}{\sin(\theta)\tan(\theta/2)^{k+m}} \eee^{i(k-m)\phi} \rightarrow 0,
\end{equation*}
at $x_{0}\rightarrow S$, provided that $k+m \ge 2$, since around $\theta=\pi$,
\begin{equation*}
\sin(\theta)\tan(\theta/2)^{k+m} \sim \frac{\sin(\theta)}{\cos(\theta/2)^{k+m}}\sim \frac{\theta-\pi}{(\frac{\theta}{2}-\pi/2)^{k+m}} \rightarrow \infty.
\end{equation*}
Therefore, we have as $x_{0}\approx S$, conditional on $P_{n}(x_{0})$,
\begin{equation*}
\nabla |P_{n}(x)|^{2}| \rightarrow 2\sqrt{n}\nabla \Re(g\overline{\xi_{1}}\overline{z}) =
2\sqrt{n}\nabla \Re(\overline{g}\xi_{1}z) = 2\sqrt{n} \left(\Re(\overline{g}\xi_{1})\nabla\Re(z)-\Im(\overline{g}\xi_{1})\nabla\Im(z) \right).
\end{equation*}
Then
\begin{equation*}
\nabla \Re(z) =
\left(\partial_{\theta}|_{\theta=\pi}, \frac{1}{\theta}\partial_{\phi}|_{\theta=\pi}\right) \left(\frac{1}{\tan(\theta/2)}\cos(\phi)\right)
= -\frac{1}{1-\cos(\theta)} \left( \cos(\phi),\sin(\phi)\right),
\end{equation*}
and
\begin{equation*}
\nabla \Im(z) =
\left(\partial_{\theta}|_{\theta=\pi}, \frac{1}{\theta}\partial_{\phi}|_{\theta=\pi}\right) \left(\frac{1}{\tan(\theta/2)}\sin(\phi)\right)
= -\frac{1}{1-\cos(\theta)} \left( \sin(\phi),-\cos(\phi)\right),
\end{equation*}
so, upon substituting, we have that, as $x_{0}\rightarrow S$,
\begin{equation*}
\begin{split}
\nabla |P_{n}(x)|^{2}| &\sim -\frac{2\sqrt{n}}{1-\cos{\theta}}\left((a\Re\xi_{1}+b\Im\xi_{1})(\cos{\phi},\sin{\phi})
-(a\Im\xi_{1}-b\Re\xi_{1}) (\sin{\phi},-\cos{\phi})\right)
\\= -\frac{2\sqrt{n}}{1-\cos{\theta}} \big(& \Re\xi_{1}(a\cos{\phi}+b\sin{\phi})+\Im\xi_{1}(b\cos{\phi}-a\sin{\phi}),\\&
\Re\xi_{1}(a\sin{\phi}-b\cos{\phi}) +\Im\xi_{1}(b\sin{\phi}+a\cos{\phi})  \big)
\\&\sim \frac{2\sqrt{n}}{1-\cos{\theta}} N\left(0, \frac{1}{2}(a^{2}+b^{2})I_{2}\right)
\end{split}
\end{equation*}
independent of $\phi$, and, by \eqref{eq:Fn norm def} and \eqref{eq:rPn(Z0,Z0)->1}, so is $|F_{n}(x)|^{2}$.
Therefore, by the continuity w.r.t.\ $x_{0}$ of the relevant random variables, the conditional distribution
of $\nabla F_{n}(S)$ on $F_{n}(S)=g$ w.r.t. an arbitrary orthonormal basis of $T_{S}(\bS^{2})$ is
that of two independent Gaussians of variance $\frac{n}{2}(a^{2}+b^{2})$.
We finally obtain \eqref{eq:neg mom 1st mix der bnd scal} by scaling the random variables $\nabla H_{n}$ by $\sqrt{n}$, and
repeating the same computation as for the limit random field on $\C$.

\vspace{2mm}

Concerning \eqref{eq:pos mom 2nd der bnd scal}, we note that, by the definition \eqref{eq:Hn norm def} of $H_{n}$,
\eqref{eq:pos mom 2nd der bnd scal}
would follow once we establish a bound for the moments of the derivatives of $\Re F_{n}$ and $\Im F_{n}$. As
the second order derivatives $V_{1}V_{2} F_{n}(S)$ are Gaussian,
to this end it is sufficient to bound from above their variances, namely prove that for all $V_{1},V_{2}\in T_{S}(\bS^{2})$
\begin{equation*}
\Var(V_{1}V_{2} F_{n}(S)) \ll n^{2}.
\end{equation*}
To this end we follow the same strategy as above, i.e.\ expand $p_{n}$ around $z=0$ as in \eqref{eq:pn def},
write $$p_{n}(z)=\Re p_{n}(z)+\Im p_{n}(z).$$ Now we differentiate \eqref{eq:pn def}, square it, and substitute $z=0$, so that only
the terms with $k=1,2$ contribute to the derivatives, provided that we check that the contribution of $z^{k}$ with
$k\ge 3$ vanish at the origin, and all the terms corresponding to $k=1,2$ are smooth at the origin.
\end{proof}

\bibliographystyle{plainnat}
\bibliography{lerario_lundberg_bib}

\end{document}